\RequirePackage{fix-cm}

\documentclass[smallextended]{svjour3}       % onecolumn (second format)

\smartqed  % flush right qed marks, e.g. at end of proof

\usepackage{etex}
 \usepackage[usenames,dvipsnames]{pstricks}
 \usepackage{epsfig}
 \usepackage {graphicx}
 \usepackage{pst-grad} % For gradients
 \usepackage{pst-plot} 
 \usepackage{pgf}
\usepackage{float}
\usepackage{tikz}
\usepackage{pgfbaseplot}
\usepackage{booktabs}
\usepackage{subcaption}
\usepackage{pstricks, pst-node, pst-plot, pst-circ}
\usepackage{moredefs}
\usepackage{array} 
\usepackage{extarrows}
\usepackage[fleqn,tbtags]{mathtools}
\usepackage{mhsetup}
\usepackage{arydshln}
\usepackage{cite}

\usepackage{amssymb,amsmath,amscd, mathrsfs, msc}
\usepackage{cases}
\usepackage{verbatim}
\usepackage[a]{esvect}
\usepackage{ntheorem}
\newcommand{\nocontentsline}[3]{}
\newcommand{\tocless}[2]{\bgroup\let\addcontentsline=\nocontentsline#1{#2}\egroup}

\theoremstyle{plain}

\theoremstyle{definition}

\theoremstyle{remark}
\numberwithin{equation}{section} \theoremstyle{corollary}

\allowdisplaybreaks

\begin{document}

\title{Global Hopf bifurcation for  differential-algebraic equations with   state dependent delay}
\author{Qingwen Hu}


\institute{Qingwen Hu \at
Department of Mathematical Sciences,  The University of Texas at Dallas,
800 West Campbell Road, Richardson, Texas, 75080 USA,\\
              \email{qingwen@utdallas.edu}           %  \\
                %      second address
}

\date{Received: date / Accepted: date}

\maketitle

\begin{abstract}

 We develop a global Hopf bifurcation theory for differential equations with a state-dependent delay  governed by an algebraic equation, using the $S^1$-equivariant degree. We apply the global Hopf bifurcation theory to  a model of genetic regulatory dynamics with threshold type state-dependent delay vanishing at the stationary state, for a description of the global continuation of the periodic oscillations. 

\end{abstract}

\section{Introduction}\label{Sec3}
Consider the following  system of  differential-algebraic equations (DAEs) with state-dependent delay,
\begin{align}\label{SDDE-general-original}
\left\{
\begin{aligned}
\dot{x}(t) & =f(x(t),\,x(t-\tau(t)),\,\sigma),\\
 \tau(t)&=g(x(t),\,x(t-\tau(t)),\,\sigma),
\end{aligned}
\right.
\end{align} 
where  we assume that
\begin{itemize}
\item[\,\,S1)] The map $f$: $\mathbb{R}^N\times\mathbb{R}^N\times\mathbb{R}\ni (\theta_1,\theta_2,\sigma) \rightarrow f(\theta_1,\theta_2,\sigma)\in\mathbb{R}^N$
is   $C^2$ (twice continuously differentiable).
\item[\,\,S2)] The  map $g$: 
$\mathbb{R}^N\times\mathbb{R}^N\times\mathbb{R}\ni (\gamma_1,\,\gamma_2,\,\sigma)\rightarrow g(\gamma_1,\,\gamma_2,\,\sigma)\in\mathbb{R}$ is $C^2$.

\item[(S3)] $(\frac{\partial}{\partial\theta_1}+\frac{\partial}{\partial\theta_2})f(\theta_1,\,\theta_2,\,\sigma)|_{\sigma=\sigma_0,\,\theta_1=\theta_2=x_{\sigma_0}}$
is nonsingular, where  $\sigma_0\in\mathbb{R}$,
$(x_{\sigma_0},\tau_{\sigma_0})$ (or, for simplicity,
$(x_{\sigma_0},\tau_{\sigma_0},\,\sigma_0)$) is a stationary state
of (\ref{SDDE-general-original}). That is,
\[
f(x_{\sigma_0},\,x_{\sigma_0},\,\sigma_0)=0, \quad
g(x_{\sigma_0},\,x_{\sigma_0},\,\sigma_0)=\tau_{\sigma_0}.
\]
\end{itemize}
(S3) implies that there exists $\epsilon_0>0$ and a
$C^1$-smooth curve
$(\sigma_0-\epsilon_0,\,\sigma_0+\epsilon_0)\ni\sigma\mapsto
(x_{\sigma},\,\tau_{\sigma})\in\mathbb{R}^{N+1}$ such that
$(x_{\sigma},\,\tau_{\sigma})$ is the unique stationary state of
(\ref{SDDE-general-original}) in a small neighborhood of
$(x_{\sigma_0},\,\tau_{\sigma_0})$ for $\sigma $ close to
$\sigma_0.$   In the following, we write $\partial_i f=\frac{\partial}{\partial \theta_i}f$ for $i=1,\,2$, and similarly we define $\partial_i g$ for $i=1,\,2$.

 The state-dependent delay of system (\ref{SDDE-general-original}) arises in several applications. To mention a few, in  the model of turning processes  \cite{Turi-1},   the delay $\tau$ is the time duration for one around of cutting; In the echo control model \cite{Walther-echo}, the state-dependent delay is the echo traveling time between the object's  positions  when the sound is emitted and received. See \cite{HKWW} for a review. To model diffusion processes in genetic regulatory dynamics with time delay, we considered in \cite{Hu-diffusion} the following system:
\begin{align}\label{SDDE-general-1}
\left\{
\begin{aligned}
\frac{\mathrm{d}x(t)}{\mathrm{d}t} & =-\mu_m x(t)+f_0(y(t-\tau(t))),\\
  \frac{\mathrm{d}y(t)}{\mathrm{d}t} & =-\mu_p y(t)+ g_0(x(t-\tau(t))),\\
 \tau(t)&=\epsilon_0+c(x(t)-x(t-\tau(t))),
\end{aligned}
\right.
\end{align}where $f_0,\,g_0: \mathbb{R}\rightarrow\mathbb{R}$ are three times continuously differentiable functions; $\mu_m$, $\mu_p$, $c$ and $\epsilon_0$ are positive constants. The time delay   $\tau(t)=\epsilon_0+c(x(t)-x(t-\tau(t)))$ models the homogenization time of the substances produced in the regulatory processes. Since the equation for $\tau$ can be written as
\begin{align}
\int_{t-\tau(t)}^t\frac{1-c\dot{x}(s)}{\epsilon_0}\mathrm{d}s=1,\label{SDDE-eqn-5}
\end{align}
we call $\tau$ a threshold type state-dependent delay and we have shown in \cite{Hu-diffusion} that   using the time transformation  $t\mapsto\int_{0}^t(1-c\dot{x}(s))\mathrm{d}s$ system~(\ref{SDDE-general-1}) can be transformed into a system with constant delay and distributed delay under certain conditions. 
In such a case, the theory we developed in \cite{HU-JDE-1} is applicable to system~(\ref{SDDE-general-1}) for a local and global  Hopf bifurcation theory. However, if $\epsilon_0=0$ in (\ref{SDDE-general-1}) and the integral equation for $\tau$ becomes
\begin{align}
\int_{t-\tau(t)}^t (1-c\dot{x}(s))\mathrm{d}s=0,\label{SDDE-eqn-degenrate}
\end{align}
which cannot be employed to remove the state-dependent delay using the time transformation  $t\mapsto \int_{0}^t(1-c\dot{x}(s))\mathrm{d}s$. Thus the global  Hopf bifurcation theory developed in  \cite{HU-JDE-1} is no longer applicable. We remark that if we obtain a differential equation of  $\tau $ from $\tau(t)=\epsilon_0+c(x(t)-x(t-\tau))$ by taking derivatives on both sides, the resulting system will have a foliation of equilibria and at least one zero eigenvalue. The global  Hopf bifurcation theory developed in  \cite{MR2644135} is not applicable  either.  With these facts, we are motivated to develop  a global Hopf bifurcation theory for   system~(\ref{SDDE-general-original}) and apply it to an extended three dimensional Goodwin's model with state-dependent delay where the delay vanishes at equilibrium. (See system~(\ref{SDDE-hes1}) at section~\ref{Regulatory-Model}
for a brief description of the model.)

To start the discussion,  we denote by
$C(\mathbb{R};\,\mathbb{R}^N)$ the normed space of bounded continuous
functions from $\mathbb{R}$ to $\mathbb{R}^N$ equipped with the
usual supremum  norm $\|x\|=\sup_{t\in\mathbb{R}}|x(t)|$ for $x\in
C(\mathbb{R};\,\mathbb{R}^N)$, where $|\cdot|$ denotes the Euclidean
norm.  We also denote by
$C^1(\mathbb{R};\,\mathbb{R}^N)$ the normed space of continuously differentiable
bounded functions with bounded derivatives from $\mathbb{R}$ to $\mathbb{R}^N$ equipped with the
usual $C^1$  norm 
\[
 \|x\|_{C^1}=\max\{\sup_{t\in\mathbb{R}}|x(t)|,\,\sup_{t\in\mathbb{R}}|\dot{x}(t)|\}\] for  $ x\in
C^1(\mathbb{R};\,\mathbb{R}^N).$

 We wish to drop the part of the algebraic equation in (\ref{SDDE-general-original}) for the application of $S^1$-equivariant degree. We have
\begin{lemma}\label{Lemma-2-1} Assume that (S2) holds. The following statements are true:
\begin{enumerate}
\item[$i)$]
For every  $(x,\,\sigma)\in C(\mathbb{R};\mathbb{R}^N)\times\mathbb{R}$, where $x$ is periodic, there exists a function  $\tau: \mathbb{R}\rightarrow\mathbb{R}$ such that  $\tau(t)=g(x(t),\,x(t-\tau(t)),\,\sigma).$ 
\item[$ii)$] Let $x_0$ be a constant function in $C(\mathbb{R};\mathbb{R}^N).$  There exists an open ball  $O(\epsilon)\subset C^1(\mathbb{R};\mathbb{R}^N)$ centered at  $x_0$, $\epsilon>0$ such that for every periodic $x\in O(\epsilon)$ with period $p>0$,  there exists a unique periodic  $\tau\in C^1(\mathbb{R};\mathbb{R})$ with period $p$ such that  $\tau(t)=g(x(t),\,x(t-\tau(t)),\,\sigma),\,t\in\mathbb{R}.$
\end{enumerate}
\end{lemma}
\begin{proof} Fix an arbitrary $t\in\mathbb{R}$ and let $a=\tau(t)$. Consider the graphs of $h=a$ and $h=g(x(t),\,x(t-a),\,\sigma)$ in the $h$-$a$ plane.  The graphs must have an intersection since $x\in C(\mathbb{R};\mathbb{R}^N)$ is periodic and $h=g(x(t),\,x(t-a),\,\sigma)$ is continuous and bounded with respect to $a$. Since $t$ is arbitrary,  there exists a function  $\tau: \mathbb{R}\rightarrow\mathbb{R}$ with $a=\tau(t)$ such that  $\tau(t)=g(x(t),\,x(t-\tau(t)),\,\sigma).$ 

Define $F:\mathbb{R}^2\rightarrow\mathbb{R}$ by $F(a,\,t)=a-g(x(t),\,x(t-a),\,\sigma)$ where  $x\in C^1(\mathbb{R};\mathbb{R}^N)$ is periodic. By (S2) $F$ is continuously differentiable in  $(a,\,t)\in\mathbb{R}$. Moreover, $a=\tau(t)$ is such that $F(a,\,t)=0$. Note that   we have
\[
\frac{\partial F}{\partial a}=1+\partial_2 g(x(t),\,x(t-a),\,\sigma)\dot{x}(t-a).
\]
Since by (S2) $\partial_2 g(\gamma_1,\,\gamma_2,\,\sigma)$ is continuous, there exists open ball   $O(\epsilon)\subset C^1(\mathbb{R};\mathbb{R}^N)$ near $x_{\sigma}$ with $\epsilon>0$, such that for every  $x\in O(\epsilon)$ with $x$ periodic, we have $\frac{\partial F}{\partial a}\neq 0$. Indeed, we can choose $\epsilon>0$ small enough so that $\frac{\partial F}{\partial a}$ assumes values in a small neighborhood of 1 in $\mathbb{R}$.
 By the Implicit Function Theorem, the solution $a=\tau(t)$ of $a=g(x(t),\,x(t-a),\,\sigma)$ for $a$ is continuously differentiable with respect to $t$. Moreover, 
 by taking derivatives on both sides of $\tau(t)=g(x(t),\,x(t-\tau(t)),\,\sigma)$ we know that $\dot{\tau}$  is bounded in $\mathbb{R}$. That is, $\tau\in C^1(\mathbb{R};\mathbb{R})$.

 Next we show that we can choose $\epsilon>0$ small enough so that $\tau$ is unique. Suppose not. Then for every $\epsilon>0$, there exists $x\in O(\epsilon)$ with period $p$ such that $\tau$ is not unique. That is, there exists $\tau_0\neq \tau$ such that $\tau_0(t)=g(x(t),\,x(t-\tau_0(t)),\,\sigma).$
 Let $\epsilon_1>0$  and $L>0$ be such that
 \[
 L=\sup_{x,\,y\in O(\epsilon_1)}\sup_{t\in\mathbb{R}}|\partial_2g(x(t),\,y(t),\,\sigma)|.
 \]
 
  Then by the Integral Mean Value Theorem, for $x\in O(\epsilon)\subset O(\epsilon_1)$, we have 
 \begin{align*}
 0< &\sup_{t\in\mathbb{R}} |\tau(t)-\tau_0(t)|\\
 =& \sup_{t\in\mathbb{R}}|g(x(t),\,x(t-\tau(t)),\,\sigma)-g(x(t),\,x(t-\tau_0(t)),\,\sigma)|\\
 =& \sup_{t\in\mathbb{R}}\left|\int_0^1\partial_2g(x(t),\,x(t)+s(x(t-\tau(t))-x(t-\tau_0(t))),\,\sigma)\mathrm{d}s(x(t-\tau(t))-x(t-\tau_0(t)))\right|\\
 =& \sup_{t\in\mathbb{R}}\left|\int_0^1\partial_2g(x(t),\,x(t)+s(x(t-\tau(t))-x(t-\tau_0(t))),\,\sigma)\mathrm{d}s\right|\\
  &\times\sup_{t\in\mathbb{R}}\left|\int_0^1\dot{x}(t+\theta(\tau(t)-\tau_0(t)))\mathrm{d}\theta (\tau(t)-\tau_0(t))\right|\\
  \leq & L \epsilon \sup_{t\in\mathbb{R}} |\tau(t)-\tau_0(t)|,
 \end{align*}which leads to $L\epsilon\geq 1$ for every $\epsilon\in (0,\,\epsilon_1)$. This is a contradiction. Therefore, $\tau$ is uniquely determined by $x$ with $\tau(t)=g(x(t),\,x(t-\tau(t)),\,\sigma),\,t\in\mathbb{R}$.
 
 Lastly, we show that $\tau$ is  $p$-periodic. Indeed,  we have for every $t\in\mathbb{R}$,
 \[
 \tau(t+p)=g(x(t+p),\,x(t+p-\tau(t+p)),\,\sigma)=g(x(t),\,x(t-\tau(t+p)),\,\sigma),
 \]which combined with $\tau(t)=g(x(t),\,x(t-\tau(t)),\,\sigma)$ and the uniqueness of the solution for $\tau$ leads to $\tau(t)=\tau(t+p),\,t\in\mathbb{R}$.
 \qed
\end{proof}

By  Lemma~\ref{Lemma-2-1}, we notice that if $x\in C(\mathbb{R};\mathbb{R}^N)$ is periodic, the function $\tau$ satisfying $\tau(t)=g(x(t),\,x(t-\tau(t)),\,\sigma)$ is not necessarily continuous and neither is $f(x(t),\,x(t-\tau(t)),\,\sigma)$, while continuity is crucial for applying topological degree theory for a Hopf bifurcation. However, if $x\in C^1(\mathbb{R};\mathbb{R}^N)$ is periodic and is in a small neighborhood of a constant function, $\tau$ is continuously differentiable. The complexity is caused by the implicitly given $\tau$ in the algebraic equation of system~(\ref{SDDE-general-original}). If we replace the delayed term $x(t-\tau(t))$ with $x(t-\tau_\sigma)$ in the algebraic equation where $\tau_\sigma$ is the stationary state of $\tau$, we obtain the following system  with state-dependent delay,
\begin{align}\label{SDDE-general}
\left\{
\begin{aligned}
\dot{x}(t) & =f(x(t),\,x(t-\tau),\,\sigma),\\
 \tau(t)&=g(x(t),\,x(t-\tau_\sigma),\,\sigma),
\end{aligned}
\right.
\end{align} where $\tau$ is continuous if $x$ is continuous.  We notice that system~(\ref{SDDE-general}) share the same set of stationary states of system~(\ref{SDDE-general-original}) and it has interest on its own right since it also represents a  class of differential-algebraic equations with state-dependent delay. Due to the similarities between  systems~(\ref{SDDE-general})  and (\ref{SDDE-general-original}), we are interested to develop global Hopf bifurcation theories for both systems, while for systems~(\ref{SDDE-general}) we use the state space of $C(\mathbb{R};\mathbb{R}^N)$ and  for system~(\ref{SDDE-general-original}) we use $C^1(\mathbb{R};\mathbb{R}^N)$. Moreover, we show that if system~(\ref{SDDE-general}) undergoes Hopf bifurcation at $(x_{\sigma_0},\,\tau_{\sigma_0})$, then system~(\ref{SDDE-general-original}) also undergoes Hopf bifurcation at the same bifurcation point. Namely, we show that systems~(\ref{SDDE-general}) and (\ref{SDDE-general-original}) share the same set of Hopf bifurcation points.

We organize the remaining part of the paper as following: Using the framework for a Hopf bifurcation theory established in  \cite{MR2644135},  we develop a local Hopf bifurcation theory for  system~(\ref{SDDE-general}) in  Section~\ref{local-bifurcations} and for   system~(\ref{SDDE-general-original}) in  Section~\ref{local-bifurcations-new}. We develop global Hopf bifurcation theories for both systems~(\ref{SDDE-general-original}) and (\ref{SDDE-general})  in Section~\ref{global-bifurcation}. In Section~\ref{Regulatory-Model} we apply the developed local and global Hopf bifurcation theories to the prototype system~(\ref{SDDE-general-1}) with $\epsilon_0=0$.  We conclude the discussion in Section~\ref{Conclude}.

%%%%%%%%%%%%%%%%%%%%%%%%%%%%%%%%%%%%%%%%%%%%%%%%%%%%%
%%%%%%%%%%%%%%%%%%%%%%%%%%%%%%%%%%
\section{Local Hopf bifurcation for system~(\ref{SDDE-general})}\label{local-bifurcations}

We begin with definitions of notations.  We denote by $V=C_{2\pi}(\mathbb{R};\,\mathbb{R}^{N})$ the space of $2\pi$-periodic continuous functions from  $\mathbb{R}$ to $\mathbb{R}^{N}$ equipped with the supremum norm. We denote by $C_{2\pi}^1(\mathbb{R};\,\mathbb{R}^{N})$ the Banach space of $2\pi$-periodic and continuously differentiable functions equipped with the $C^1$ norm.  

Note that if $x\in C(\mathbb{R};\mathbb{R}^N)$ is $p$-periodic, then  $\tau(t)=g(x(t),\,x(t-\tau_\sigma),\,\sigma),\,t\in\mathbb{R}$ is continuous and $p$-periodic. We call $x$ a solution if   $(x,\,\tau)$ satisfies system~(\ref{SDDE-general}). For a stationary state $x_0$ of system~(\ref{SDDE-general}) with the
parameter $\sigma _0$, we say that $(x_0,\,\sigma_0)$ is a Hopf
bifurcation point of system (\ref{SDDE-general}), if there exist a sequence
$\{(x_k,\,\sigma_k,\,T_k)\}_{k=1}^{+\infty}\subseteq C(\mathbb{R};
\mathbb{R}^{N})\times\mathbb{R}^2$ and $T_0> 0$ such that
$$
\lim_{k\rightarrow+\infty}\|(x_k,\,\sigma_k,\,
T_k)-(x_0,\,\sigma_0,\,T_0)\|_{C(\mathbb{R};
\mathbb{R}^{N})\times\mathbb{R}^2}=0,
$$
and $(x_k,\,\sigma_k)$ is a nonconstant $T_k$-periodic solution of
system (\ref{SDDE-general}).

Due to the nature of the same approach of using the $S^1$-equivariant degree, the presentation of the remaining part of this section is similar to that of \cite{MR2644135}, even though the systems in question are different. We   study  Hopf
bifurcation of (\ref{SDDE-general}) through the system obtained through the formal linearization \cite{cooke1996problem}. Namely, we freeze the state-dependent delay in system~(\ref{SDDE-general}) at its stationary state  and linearize the resulting differential equation of $x$ with constant delay at the stationary state.
For $\sigma\in
(\sigma_0-\epsilon_0,\,\sigma_0+\epsilon_0)$,  the following system is called the formal linearization of  system (\ref{SDDE-general}) at the stationary point
$x_{\sigma}$:
\begin{align}\label{eq-3-1}
\dot{x}(t) =  
\partial_1 f(\sigma)  \left( 
 x(t)-x_\sigma 
\right) + 
\partial_2 f(\sigma) \left(  x(t-\tau_\sigma)-x_\sigma
\right),
\end{align}
where
\begin{align*}
\partial_1 f(\sigma)&:=\partial_1 f(x_\sigma,
\,\tau_\sigma,\,\sigma), \,
\partial_2 f(\sigma):=\partial_2 f(x_\sigma,
\,\tau_\sigma,\,\sigma),\,\tau_\sigma=g(x_\sigma,\,x_\sigma,\,\sigma).
\end{align*}
Letting  $x(t)=e^{\omega t}\cdot C+x_\sigma$ with $C\in\mathbb{R}^{N}$, we obtain the following characteristic equation of the linear system corresponding to the inhomogeneous linear system (\ref{eq-3-1}),
\begin{align}\label{eq-3-2}
\det\Delta_{(x_\sigma,\,\sigma)}(\omega)=0,
\end{align}
where $\Delta_{(x_\sigma,\,\sigma)}(\omega)$ is an $N\times N$
complex matrix defined by
\begin{align}\label{character-matrix}
\Delta_{(x_\sigma,\,\sigma)}(\omega)=\omega I- \partial_1 f(\sigma) - \partial_2 f(\sigma)  e^{-\omega\tau_{\sigma}}.
\end{align}
A solution $\omega_0$ to the characteristic equation (\ref{eq-3-2}) is called
a characteristic value of the stationary state
$(x_{\sigma_0},\,\sigma_0)$. If  zero is not a characteristic value of $(x_{\sigma_0},\,\sigma_0)$, $(x_{\sigma_0},\, \sigma_0)$ is said to be a nonsingular stationary state. We say that $(x_{\sigma_0},\,\sigma_0)$ is a \textit{center} if the set of nonzero purely imaginary characteristic values of $(x_{\sigma_0},\,\sigma_0)$  is nonempty and discrete.  $(x_{\sigma_0},\,\sigma_0)$ is called an \emph{isolated center} if it is the only center in some neighborhood of $(x_{\sigma_0},\,\sigma_0)$ in $\mathbb{R}^{N}\times\mathbb{R}$.

If $(x_{\sigma_0}, \,\sigma_0)$ is
an isolated center of (\ref{eq-3-1}), then there exist $\beta_0>0$ and
 $\delta\in (0,\,\epsilon_0)$  such that
\[ \det \Delta_{ (x_{\sigma_0}, \,\sigma_0)}(i\beta_0)=0,
\]
and
\begin{align}\label{character-pure}
 \det \Delta_{ (x_{\sigma}, \,\sigma)}(i\beta)\neq 0,
\end{align}
 for any  $\sigma\in (\sigma_0-\delta,\,\sigma_0+\delta)$ and any $\beta\in (0,\,+\infty)\setminus\{\beta_0\}$.
Hence, we can
choose constants $\alpha_0=\alpha_0(\sigma_0,\,\beta_0)>0$ and
$\varepsilon=\varepsilon(\sigma_0,\,\beta_0)>0$ such that the
closure of the set $ \Omega:=(0,\,\alpha_0)\times
(\beta_0-\varepsilon,\,\beta_0+\varepsilon)\subset\mathbb{R}^2\cong\mathbb{C}$
contains no other zero of $\det \Delta_{
(x_{\sigma_0},\, \sigma_0)}(\cdot)$ in $\partial\Omega$.  We note that $\det \Delta_{ (x_{\sigma}, \, \sigma)}(\omega )$  is analytic in $\omega$ and is continuous in $\sigma$. If $\delta>0$ is small enough, then there is no zero of $\det \Delta_{(x_{{\sigma_0\pm\delta}}, \, {\sigma_0\pm\delta})}(\omega )$ in $\partial\Omega$.  So
we can define the number
\begin{align*}
\gamma_{\pm}(x_{\sigma_0},\, \sigma_0,
\,\beta_0)=\deg_B (\det \Delta_{(x_{{\sigma_0\pm\delta}}, \, {\sigma_0\pm\delta})}(\cdot ),\,\Omega),
\end{align*}
and the  crossing number of
$(x_{\sigma_0},\, \sigma_0, \,\beta_0)$  as
\begin{align}\label{crossingnumber}
\gamma (x_{\sigma_0},\,  \sigma_0,
\,\beta_0)=\gamma_{-}-\gamma_{+},
\end{align} where $\deg_B$ is the Brouwer degree in finite-dimensional spaces. See, e.g., \cite{kw}, for details. 

To formulate the Hopf bifurcation problem as a fixed point problem
in $C_{2\pi}(\mathbb{R};\mathbb{R}^N)$, we  normalize
the period of the $2\pi/\beta$-periodic solution $x$ of (\ref{SDDE-general}) and the associated $\tau\in C(\mathbb{R};\mathbb{R})$ by setting
$(x(t),\tau(t))=(y(\beta t),\,z(\beta t))$ and obtain
\begin{align}\label{eq-3-4}
\left(
\begin{array}{c}
\dot{y}(t)\\
 {z}(t)
\end{array}
\right)=\left(
\begin{array}{c}
\frac{1}{\beta}f(y(t),\,y(t-\beta z(t)),\,\sigma)\\
g(y(t),\,y(t-\beta z_\sigma),\,\sigma)
\end{array}
\right),
\end{align}where $(y_{\sigma},\,z_{\sigma})=(x_{\sigma},\,\tau_{\sigma})$.

%In the following presentation, we drop the algebraic equation of system~(\ref{eq-3-4}) and  call 
%$y\in C_{2\pi}(\mathbb{R};\mathbb{R}^N)$ a solution.

Define $N_0: V\ni(y,\,\sigma,\,\beta) \times\mathbb{R}^2\rightarrow N_0(y,\,\sigma,\,\beta)\in V$   by
\begin{align}\label{N-0-definition}
 N_0(y,\,\sigma,\,\beta)(t)= f(y(t),\,y(t-\beta z(t)),\,\sigma). 
\end{align}Then the part of differential equations of system~(\ref{eq-3-4}) is rewritten as
\begin{align}\label{eq-3-4-new}
\dot{y}(t) = \frac{1}{\beta}N_0(y,\,\sigma,\,\beta)(t). 
\end{align}
Correspondingly, (\ref{eq-3-1}) is transformed into
\begin{align}\label{eq-3-5}
\dot{y}(t) =\frac{1}{\beta}\tilde{N}_0(y,\,\sigma,\,\beta)(t),
\end{align}
where $\tilde{N}_0:  V\ni(y,\,\sigma,\,\beta) \times\mathbb{R}^2\rightarrow \tilde{N}_0(y,\,\sigma,\,\beta)\in V$ is defined by
\begin{align*}
\tilde{N}_0(y,\,\sigma,\,\beta)(t)=  \partial_1 f(\sigma)   \left(y(t)-y_\sigma\right)
+  
\partial_2 f(\sigma) \left(y(t-\beta z_\sigma)-y_\sigma\right).
\end{align*}
We note that $y$ is $2\pi$-periodic if and only if $x$
is $(2\pi/\beta)$-periodic. 
 
%$S^1\backsimeq\mathbb{R}/2\pi\mathbb{Z}$ and $V=C(S^1;\,\mathbb{R}^{N+1})\backsimeq C(\mathbb{R}/2\pi\mathbb{Z};\,\mathbb{R}^{N+1})$. That is, we identify a %2$\pi$-periodic function $u:\mathbb{R}\rightarrow\mathbb{R}^{N+1}$ with the  function $\tilde{u}: S^1\rightarrow\mathbb{R}^{N+1}$ via $u(t)=\tilde{u}(e^{it})$.

Let $L_0: C_{2\pi}^1(\mathbb{R};\,\mathbb{R}^{N})\rightarrow V$ be defined by $L_0
y(t)=\dot{y}(t), t\in \mathbb{R}$ and $K: V\rightarrow \mathbb{R}^{N}$
 be defined by 
 \begin{align}\label{K-map}
 K(y)=\frac{1}{2\pi}\int_0^{2\pi}y(t) \mathrm{d} t.
 \end{align}

Define the map $\tilde{\mathcal{F}}:
V\times\mathbb{R}^2\rightarrow V$ by
\begin{align}\label{tilde-F}
&\widetilde{\mathcal{F}}(y,\,\sigma,\,\beta):=y-(L_0+K)^{-1}
\left[\frac{1}{\beta} \tilde{N}_0(y, \,\sigma,\,\beta)+K (y)\right].
\end{align}  We call the set defined by
\[
B(y_0,\,\sigma_0,\,\beta_0; r,\,\rho)=\{(y,\,\sigma,\,\beta):
\|y-y_\sigma\|<r,\,|(\sigma,\,\beta)-(\sigma_0,\,\beta_0)|<\rho\},
 \]  a {\it special neighborhood } of $\widetilde{\mathcal{F}}$, if it
   satisfies 
  \begin{enumerate}
  \item[i)] 
   $\widetilde{\mathcal{F}}(y,\,\sigma,\,\beta)\neq 0$ for every $(y,\,\sigma,\,\beta)\in \overline{B(y_0,\,\sigma_0,\,\beta_0; r,\,\rho)}$ with $|(\sigma,\,\beta)-(\sigma_0,\,\beta_0)|=\rho$ and $\|y-y_\sigma\|\neq 0$;
   \item[ii)] $(y_0,\,\sigma_0,\,\beta_0)$ is the only isolated center in 
$\overline{B(y_0,\,\sigma_0,\,\beta_0; r,\,\rho)}$. 
\end{enumerate}
Before we state and prove our local Hopf bifurcation theorem, we
cite some technical  Lemmas   from \cite{MR2644135} with necessary notational adaptions.
\begin{lemma}[\cite{MR2644135}\,]\label{bounded-inverse}
Let $L_0: C_{2\pi}^1(\mathbb{R};\,\mathbb{R}^{N})\rightarrow V$ be defined by $L_0
y(t)=\dot{y}(t), t\in \mathbb{R}$ and let $K: V\rightarrow \mathbb{R}^{N}$
 be defined at (\ref{K-map}). Then $L_0+K$ has a compact inverse
  $(L_0+K)^{-1}: V\rightarrow V$.
\end{lemma} 
\begin{lemma}[\cite{MR2644135}\,]\label{bounded-N-0}
For any $\sigma\in\mathbb{R}$ and $\beta>0$, the map $N_0(\cdot,\,\sigma,\,\beta): V\rightarrow V$ defined by (\ref{N-0-definition}) is continuous.
\end{lemma} 
\begin{lemma}[\cite{MR2644135}\,]\label{linear-eq-center}%Assume $($S3$\,)$ holds.
If system $($\ref{eq-3-1}$\,)$ has a nonconstant periodic
solution  with  period $T>0$, then there exists an integer $m\geq 1,\,m\in\mathbb{N}$ such that $\pm im\,2\pi/T$ are characteristic values of the stationary state
$(x_{\sigma},\,\tau_{\sigma},\,\sigma)$.
\end{lemma} 
For the purpose of establishing the $S^1$-degree on some special neighborhood near the stationary state, we have
\begin{lemma}\label{lemma-3-1}
Assume $($\mbox{{S1}}$\,)$--$($\mbox{{S3}}$\,)$ hold. Let $L_0$ and $K$ be as in
Lemma~\ref{bounded-inverse} and $\tilde{N}_0:V\times\mathbb{R}^2\rightarrow V$ be as in
$($\ref{eq-3-5}$\,)$.  Let $\tilde{\mathcal{F}}:
V\times\mathbb{R}^2\rightarrow V$ be defined at (\ref{tilde-F}).  If $B(y_0,\,\sigma_0,\,\beta_0; r,\,\rho)$ is a special neighborhood of $\widetilde{\mathcal{F}}$    with $0<\rho<\beta_0$, then there exists
$r'\in (0,\, r]$ such that the neighborhood
\[
B(y_0,\,\sigma_0,\,\beta_0;
r',\,\rho)=\{(u,\,\sigma,\,\beta):
\|y-y_\sigma\|<r',\,|(\sigma,\,\beta)-(\sigma_0,\,\beta_0)|<\rho\}
\] satisfies
\[
\dot{y}(t)\not\equiv \frac{1}{\beta}  f(y(t),y(t-\beta z(t)),\,\sigma)
\]
for $(y,\,\sigma,\,\beta)\in\overline{B(y_0,\,\sigma_0,\,\beta_0;
r',\,\rho)}$ with $y\neq y_\sigma$ and
$|(\sigma,\,\beta)-(\sigma_0,\beta_0)|=\rho$.
\end{lemma}
\begin{proof}We prove by contradiction. 
Suppose the statement is not true, then for any $0<r'\leq r$, there exists
$(y,\,\sigma,\,\beta)$ such that
$0<\|y-y_\sigma\|<r',\,|(\sigma,\,\beta)-(\sigma_0,\,\beta_0)|=\rho$
and
\begin{align}\label{eq-3-13}
\dot{y}(t)=\frac{1}{\beta} 
f(y(t),y(t-\beta z(t)),\,\sigma) \textrm{ for } t\in\mathbb{R}.
\end{align}
Then there exists a sequence of nonconstant periodic solutions
$\{(y_k,\,\sigma_k,\,\beta_k)\}_{k=1}^{\infty}$
of (\ref{eq-3-13})
 such that
\begin{align}\label{eq-3-14}
&\lim_{k\rightarrow +\infty}\|y_k-y_{\sigma_k}\|=0,\,\,
|(\sigma_k,\,\beta_k)-(\sigma_0,\,\beta_0)|=\rho, \intertext{and}
&\dot{y}_k(t)=\frac{1}{\beta_k} 
f(y_k(t),y_k(t-\beta_k z_k(t)),\,\sigma_k) 
 \textrm{ for } t\in\mathbb{R},\label{eq-3-15}
\end{align}where $z_k$ is chosen according to $y_k$ in light of Lemma~\ref{Lemma-2-1} so that $(y_k,\,z_k)$ is a solution of system~(\ref{eq-3-4}).

Note that $0<\rho<\beta_0$ implies that $\beta_k\geq \beta_0-\rho>0$
for every $k\in\mathbb{N}$. Also, since the sequence
$\{\sigma_k,\,\beta_k\}_{k=1}^{\infty}$ belongs to a bounded
neighborhood of $(\sigma_0,\,\beta_0)$ in $\mathbb{R}^2$, there
exists a convergent subsequence, still denoted by
$\{(\sigma_k,\,\beta_k)\}_{k=1}^{\infty}$ for notational simplicity, that converges to
$(\sigma^*,\,\beta^*)$ so that
$|(\sigma^*,\,\beta^*)-(\sigma_0,\,\beta_0)|=\rho$ and $\beta^*>0$.  Then we have
\begin{align}\label{eq-3-16}
&\lim_{k\rightarrow +\infty}\|y_k-y_{\sigma_k}\|=0,\,\,
\lim_{k\rightarrow+\infty}|(\sigma_k,\,\beta_k)-(\sigma^*,\,\beta^*)|=0,
\intertext{and} &\dot{y}_k(t)=\frac{1}{\beta_k} 
f(y_k(t),y_k(t-\beta_k z_k(t)),\,\sigma_k)  \textrm{ for } t\in\mathbb{R}.\label{eq-3-17}
\end{align}
In the following we show that the system
\begin{align}
\dot{v}(t)=\frac{1}{\beta^*} \partial_1 f(\sigma^*) v(t)+\frac{1}{\beta^*} \partial_2 f(\sigma^*)  v(t-\beta^* z_{\sigma^*}),\label{eq-3-18}
\end{align}
has a nonconstant periodic solution which contradicts the
assumption that $(y_{\sigma_0},\,\sigma_0,\,\beta_0)$ is the only
center of (\ref{eq-3-5}) in $\overline{B(u_0,\,\sigma_0,\,\beta_0;
r,\,\rho)}$.

By (S1), $f$: $\mathbb{R}^N\times\mathbb{R}^N\times\mathbb{R}\ni (\theta_1,\theta_2,\sigma) \rightarrow f(\theta_1,\theta_2,\sigma)\in\mathbb{R}^N$
is $C^2$  in $(\theta_1,\,\theta_2)$. It follows
from the Integral Mean Value Theorem   that
\begin{align}
\dot{y}_k(t)=&\frac{1}{\beta_k}\int_0^1 \partial_1 f_k(\sigma_k,\,s)(t)  \mathrm{d}s(y_k(t)-y_{\sigma_k})\notag\\
& +\frac{1}{\beta_k}\int_0^1 
\partial_2 f_k(\sigma_k,\,s)(t)  \mathrm{d}s (y_k(t-\beta_k z_k(t))-y_{\sigma_k}),\label{eq-3-19}
\end{align}
where
\begin{align*}
\partial_1 f_k(\sigma_k,\,s)(t):&=
\partial_1 f(y_{\sigma_k}+s(y_k(t)-y_{\sigma_k}),y_{\sigma_k}+s(y_k(t-z_k(t))-y_{\sigma_k}),\,\sigma_k)
),\\
\partial_2 f_k(\sigma_k,\,s)(t):&=
\partial_2 f(y_{\sigma_k}+s(y_k(t)-y_{\sigma_k}),y_{\sigma_k}+s(y_k(t-z_k(t))-y_{\sigma_k}),\,\sigma_k)
).
\end{align*}
Put
\begin{align}{\label{eq-3-20}}
v_k(t)=\frac{y_k(t)-y_{\sigma_k}}{\|y_k-y_{\sigma_k}\|}.
\end{align}
Then  we have
\begin{align}{\label{eq-3-21}}
v_k(t-\beta_k z_k(t))=\frac{y_k(t-\beta_k
z_k(t))-y_{\sigma_k}}{\|y_k-y_{\sigma_k}\|}.
\end{align}
By (\ref{eq-3-19}) and (\ref{eq-3-21}) we have
\begin{align}
\dot{v}_k(t)=&\frac{1}{\beta_k}\int_0^1  \partial_1 f_k(\sigma_k,\,s)(t)  \mathrm{d}s\,v_k(t)+\frac{1}{\beta_k}\int_0^1 \partial_2 f_k(\sigma_k,\,s)(t)  \mathrm{d}s\, v_k(t-\beta_k z_k(t)).\label{eq-3-22}
\end{align}
We claim that there exists a convergent subsequence of
$\{v_k\}_{k=1}^{+\infty}$. Indeed, by (\ref{eq-3-14}) and system~(\ref{eq-3-4}), we know that $\{z_k,\,\beta_k\}_{k=1}^{+\infty}$ is uniformly bounded in $C(\mathbb{R};\mathbb{R})\times\mathbb{R}$ and hence
 $\lim_{t\rightarrow+\infty} [t-\beta_k z_k(t)]
=+\infty$. Then by (\ref{eq-3-20}) and
(\ref{eq-3-21}), we have
\[
\|v_k\|=1,\,\|v_k(\cdot-\beta_k
z_k(\cdot))\|=1.
\]
Recall that $\partial_i f(\sigma^*)$ and $\partial_i g(\sigma^*)$, $i=1,\,2$, are defined in (\ref{eq-3-1}).
By (\ref{eq-3-16}), we know that $(y_{\sigma_k}+s(y_k(t)-y_{\sigma_k}),y_{\sigma_k}+s(y_k(t-z_k(t))-y_{\sigma_k}),\,\sigma_k)$ converges to the stationary state $(y_{\sigma^*},\,y_{\sigma^*},\,\sigma^*)$ in $C(\mathbb{R};\mathbb{R}^{2N})\times\mathbb{R}$  uniformly for all   $s \in [0,1 ]$.
By (S1) we know that $f(\theta_1,\,\theta_2,\,\sigma)$ is $C^2$ in $(\theta_1,\,\theta_2,\,\sigma)$ and $\partial_1
f(\theta_1,\,\theta_2,\,\sigma)$ is $C^1$ in $\sigma$. Also, by (\ref{eq-3-14}), the sequence $\{u_k,\,\beta_k,\,\sigma_k\}_{k=1}^{+\infty}$ is uniformly bounded in $C(\mathbb{R};\mathbb{R}^{N+1})\times\mathbb{R}^2$.  Then there exists a constant $\tilde{L}_1>0$ so that
\begin{align*}
&\,\,|\partial_1
f_k(\sigma_k,\,s)(t)-\partial_1
f(\sigma^*)|\\
\leq &\,\, \tilde{L}_1 |(y_{\sigma_k}+s(y_k(t)-y_{\sigma_k}),y_{\sigma_k}+s(y_k(t-z_k(t))-y_{\sigma_k}),\,\sigma_k)-  (y_{\sigma^*},\,y_{\sigma^*},\,\sigma^*)|,
\end{align*}
for all $t\in\mathbb{R}$, $k\in\mathbb{N}$ and $s\in [0,\,1]$. Therefore, we have $\lim_{k\rightarrow+\infty}\|\partial_1 f_k(\sigma_k,\,s)-\partial_1
f(\sigma^*)\|=0$ uniformly for $s\in [0, 1]$. By the same argument we obtain that
\begin{align}\label{uniform-convergence-a} 
%\left\{
%\begin{array}{l}
\lim_{k\rightarrow+\infty}\|\partial_1 f_k(\sigma_k,\,s)-\partial_1
f(\sigma^*)\|=0,\,\,
\lim_{k\rightarrow+\infty}\|\partial_2 f_k(\sigma_k,\,s)-\partial_2
f(\sigma^*)\|=0, 
%\end{array}
%\right.
\end{align}
uniformly for $s\in [0, 1]$. From (\ref{uniform-convergence-a}) we know that
$\|\partial_1 f_k(\sigma_k,\,s)\|$ and $\|\partial_2 f_k(\sigma_k,\,s)\|$ are both uniformly bounded for all
$k\in\mathbb{N}$ and $s\in [0,\,1]$. Then it follows from
(\ref{eq-3-22}) that there exists a constant $\tilde{L}_2>0$ such that $\|\dot{v}_k\|<\tilde{L}_2$
for any $k\in\mathbb{N}$. By the Arzela-Ascoli Theorem, there exists a
convergent subsequence $\{v_{k_j}\}_{j=1}^{+\infty}$ of
$\{v_k\}_{k=1}^{+\infty}$. That is, there exists $v^*\in\{v\in V:
\|v\|=1\}$ such that
\begin{align}{\label{eq-3-23}}
\lim_{j\rightarrow+\infty}\|v_{k_j}-v^*\|=0.
\end{align}
By the Integral Mean Value Theorem, we have
\begin{align}
& |v_{k_j}(t-\beta_{k_j} z_{k_j}(t))-v_{k_j}(t-\beta^* z_{\sigma^*})|\notag\\
= & \left|\int_0^1 \dot{v}_{k_j}(t-\theta(\beta_{k_j} z_{k_j}(t)-\beta^*
z_{\sigma^*}))\mathrm{d}\theta(\beta_{k_j} z_{k_j}(t)-\beta^*
z_{\sigma^*})\right|\notag\\
\leq & \|\dot{v}_{k_j}\|\cdot |\beta_{k_j} z_{k_j}(t)-\beta^*
z_{\sigma^*}|\notag\\
\leq & \tilde{L}_2 (\beta_{k_j}
|z_{k_j}(t)-z_{\sigma^*}|+|\beta_{k_j}-\beta^*|z_{\sigma^*}).\label{eq-3-24}
\end{align}
 By (\ref{eq-3-16}) and (\ref{eq-3-24}) we
have
\begin{align}\label{eq-3-25}
\lim_{j\rightarrow+\infty}\|v_{k_j}(\cdot-\beta_{k_j} z_{k_j}(\cdot))-v_{k_j}(\cdot-\beta^*
z_{\sigma^*})\|=0.
\end{align}
Therefore, it follows from (\ref{eq-3-23}) and  (\ref{eq-3-25}) that
\begin{align}\label{eq-3-26}
\lim_{j\rightarrow+\infty}\|v_{k_j}(\cdot-\beta_{k_j} z_{k_j}(\cdot))-v^*(\cdot-\beta^*
z_{\sigma^*})\|=0.
\end{align}
It follows from (\ref{eq-3-16}), (\ref{uniform-convergence-a}), (\ref{eq-3-23}) and (\ref{eq-3-26})
that the right hand side of (\ref{eq-3-22}) converges uniformly to
the right hand side of (\ref{eq-3-18}). Therefore, $v^*$ is differentiable and  we have
\begin{align*}
\lim_{k\rightarrow+\infty}|\dot{v}_k(t)-\dot{v}^*(t)|=0,
\end{align*}
and
\begin{align}
\dot{v}^*(t)=\frac{1}{\beta^*} \partial_1 f(\sigma^*) v^*(t)+\frac{1}{\beta^*} \partial_2 f(\sigma^*)  v^*(t-\beta^* z_{\sigma^*}).\label{eq-3-27}
\end{align}
Since by (S3) the
matrix
$
\partial_1 f(\sigma^*)+\partial_2 f(\sigma^*),
 $ is nonsingular, $v=0$ is the only constant solution of
(\ref{eq-3-27}). Also, we have $v^*\in \{v\in V: \|v\|=1\}$,  $\|v^*\|\neq 0$. Therefore,
$(v^*(t),\,\sigma^*,\,\beta^*)$ is a nonconstant
periodic solution of the linear equation (\ref{eq-3-27}). Then by Lemma~\ref{linear-eq-center}
$(y_{\sigma^*},\,\sigma^*,\,\beta^*)$ is also a center of
(\ref{eq-3-5})  in $\overline{B(y_0,\,\sigma_0,\,\beta_0;
r,\,\rho)}$. This contradicts the assumption that
$B(y_0,\,\sigma_0,\,\beta_0; r,\,\rho)$ is a special neighborhood
of (\ref{eq-3-4}). This completes the proof.\hfill{ }\qed
\end{proof}
To apply the homotopy argument of $S^1$-degree, we show the following
\begin{lemma}\label{homotopy-lemma}
Assume (S1)--(S3) hold. Let $L_0$, $K$, $\tilde{N}_0$, $\tilde{\mathcal{F}}$ be as in Lemma~\ref{lemma-3-1}  and  $N_0: V\times\mathbb{R}^2\rightarrow V$ be as in $($\ref{eq-3-4}$\,)$. Define the map $\mathcal{F}:
V\times\mathbb{R}^2\rightarrow V$  by
\begin{align*}
&\mathcal{F}(y,\,\sigma,\,\beta):=y-(L_0+K)^{-1}\left[\frac{1}{\beta}N_0(y,\,\sigma,\,\beta)+K(y)\right].
\end{align*}
If $\mathcal{U}=B(y_0,\,\sigma_0,\,\beta_0; r,\,\rho)$ $\subseteq
V\times\mathbb{R}^2$ is a special neighborhood of
$\tilde{\mathcal{F}}$  with $0<\rho<\beta_0$, then there exists
$r'\in (0,\,r]$ such that
$\mathcal{F}_{\theta}=(\mathcal{F},\,\theta)$ and
$\tilde{\mathcal{F}}_{\theta}=(\tilde{\mathcal{F}},\,\theta)$  are
homotopic on $\overline{B(y_0,\,\sigma_0,\,\beta_0; r',\,\rho)}$,
where $\theta$ is a completing function (or Ize's function) defined on
$\overline{B(y_0,\,\sigma_0,\,\beta_0; r',\,\rho)}$ which satisfies
\begin{enumerate}
\item[i)] $\theta(y_\sigma,\,\sigma,\,\beta)=-|(\sigma,\,\beta)-(\sigma_0,\,\beta_0)|$ if $(y_\sigma,\,\sigma,\,\beta)\in \bar{\mathcal{U}};$
\item[ii)] $\theta(y,\,\sigma,\,\beta)=r'$ if $\|y-y_\sigma\|=r'$.
\end{enumerate}
\end{lemma}
\begin{proof} Since $\mathcal{U}=B(y_0,\,\sigma_0,\,\beta_0; r,\,\rho)$ $\subseteq
V\times\mathbb{R}^2$ is a special neighborhood of
$\tilde{\mathcal{F}}$  with $0<\rho<\beta_0$, then by Lemma~\ref{lemma-3-1}, both $\mathcal{F}_{\theta}=(\mathcal{F},\,\theta)$ and
$\tilde{\mathcal{F}}_{\theta}=(\mathcal{F},\,\theta)$ are $\mathcal{U}$-admissible. That is, the $S^1$
degrees of $\mathcal{F}_{\theta}$ and $\tilde{\mathcal{F}}_{\theta}$ are well-defined on  $\mathcal{U}$.

Suppose, for contradiction, that the conclusion is not true. Then for any $r'\in (0,\,r]$,
$\mathcal{F}_{\theta}=(\mathcal{F},\,\theta)$ and
$\tilde{\mathcal{F}}_{\theta}=(\tilde{\mathcal{F}},\,\theta)$  are not homotopic
on $\overline{B(y_0,\,\sigma_0,\,\beta_0; r',\,\rho)}$. That is, any homotopy map between
$\mathcal{F}_{\theta}$ and $\tilde{\mathcal{F}}_{\theta}$ has a zero on the boundary of $\overline{B(y_0,\,\sigma_0,\,\beta_0; r',\,\rho)}$. In particular, the linear homotopy $h(\cdot,\,\alpha):=\alpha \mathcal{F}_{\theta}+(1-\alpha)\tilde{\mathcal{F}}_{\theta}=(\alpha \mathcal{F} +(1-\alpha)\tilde{\mathcal{F}},\theta) $ has a zero on the boundary of $\overline{B(y_0,\,\sigma_0,\,\beta_0; r',\,\rho)}$, where $\alpha\in [0,\,1]$.

Note that $\theta(y,\,\sigma,\,\beta)>0$ if
$\|y-y_\sigma\|=r'$.
Then, there exist
$(y,\,\sigma,\,\beta)$ and $\alpha\in [0,\,1]$ such that
$\|y-y_\sigma\|<r',\,|(\sigma,\,\beta)-(\sigma_0,\,\beta_0)|=\rho$
and
\begin{align}
H(y,\,\sigma,\,\beta,\,\alpha):=\alpha \mathcal{F} +(1-\alpha)\tilde{\mathcal{F}}=0.\label{eq-3-28}
\end{align}
Since $r'>0$ is arbitrary in the interval $(0,\,r]$, there exists a
nonconstant sequence
$\{(y_k,\,\sigma_k,\,\beta_k,\,\alpha_k)\}_{k=1}^{\infty}$ of
solutions of (\ref{eq-3-28}) such that
\begin{align}
&      \lim_{k\rightarrow +\infty}\|y_k-y_{\sigma_k}\|=0,\,\,
       |(\sigma_k,\,\beta_k)-(\sigma_0,\,\beta_0)|=\rho,\, 0
       \leq \alpha_k\leq 1,
       \label{eq-3-29}
\intertext{and} &      H(y_k,\,\sigma_k,\,\beta_k,\,\alpha_k)=0, \mbox{ for all } k\in\mathbb{N}.
\label{eq-3-30}
\end{align}
Note that $0<\rho<\beta_0$ implies that $\beta_k\geq \beta_0-\rho>0$
for every $k\in\mathbb{N}$. From (\ref{eq-3-29}) we know that
$\{(\sigma_k,\,\beta_k,\,\alpha_k)\}_{k=1}^{\infty}$ belongs to a
compact subset of $\mathbb{R}^3$. Therefore, there exist  a
convergent subsequence, denoted for notational simplicity by
$\{(\sigma_k,\,\beta_k,\,\alpha_k)\}_{k=1}^{\infty}$ without loss of
generality, and $(\sigma^*,\,\beta^*,\,\alpha^*)\in \mathbb{R}^3$
such that $\beta^*\geq \beta_0-\rho>0$, $\alpha^*\in [0,\,1]$ and
\begin{align}\label{eq-3-31}
\lim_{k\rightarrow+\infty}
|(\sigma_k,\,\beta_k,\,\alpha_k)-(\sigma^*,\,\beta^*,\,\alpha^*)|=0.
\end{align}
By the same token for the proof of Lemma~\ref{lemma-3-1}, we show that the
system
\begin{align}
\dot{v}(t)=\frac{1}{\beta^*} \partial_1 f(\sigma^*)  v(t)+\frac{1}{\beta^*} \partial_2 f(\sigma^*) v(t-\beta^* z_{\sigma^*})\label{eq-3-32}
\end{align}
with $\partial_i f(\sigma^*),\, \partial_i g(\sigma^*),  i=1,\,2,$ defined at (\ref{eq-3-1}), has a nonconstant periodic solution which contradicts the
assumption that $B(u_0,\,\sigma_0,\,\beta_0; r,\,\rho)$ is a special neighborhood which contains an
isolated center of (\ref{eq-3-5}).

By (\ref{eq-3-30}), we know that the subsequence
$\{(y_k,\,\sigma_k,\,\beta_k,\,\alpha_k)\}_{k=1}^{\infty}$
satisfies
\begin{align}\label{eq-3-33}
H(y_k,\,\sigma_k,\,\beta_k,\,\alpha_k)=0.
\end{align}
By (S1), $f$: $\mathbb{R}^N\times\mathbb{R}^N\times\mathbb{R}\ni (\theta_1,\theta_2,\sigma) \rightarrow f(\theta_1,\theta_2,\sigma)\in\mathbb{R}^N$
is $C^2$  in $(\theta_1,\,\theta_2)$. Then it follows
from the Integral Mean Value Theorem and from (\ref{eq-3-33}) that
\begin{align}
\dot{u}_k(t)=&\frac{\alpha_k}{\beta_k}\int_0^1  \partial_1 f_k(\sigma_k,\,s)(t)  \mathrm{d}s (y_k(t)-y_{\sigma_k}) \notag\\
&+\frac{\alpha_k}{\beta_k}\int_0^1  \partial_2 f_k(\sigma_k,\,s)(t)  \mathrm{d}s 
(y_k(t-\beta_k z_k(t))-y_{\sigma_k}) \notag\\
&+ \frac{1-\alpha_k}{\beta_k}\int_0^1 \partial_1 f_k(\sigma_k,\,s)(t) \mathrm{d}s(y_k(t)-y_{\sigma_k})\notag\\
& +\frac{1-\alpha_k}{\beta_k} \int_0^1\partial_2 f_k(\sigma_k,\,s)(t)\mathrm{d}s (y_k(t-\beta_k z_{\sigma_k})-y_{\sigma_k}), \label{eq-3-34}
\end{align}
where
\begin{align*}
\partial_1 f_k(\sigma_k,\,s)(t):&=
\partial_1 f(y_{\sigma_k}+s(y_k(t)-y_{\sigma_k}),y_{\sigma_k}+s(y_k(t-\beta z_k(t))-y_{\sigma_k}),\,\sigma_k)
),\\
\partial_2 f_k(\sigma_k,\,s)(t):&=
\partial_2 f(y_{\sigma_k}+s(y_k(t)-y_{\sigma_k}),y_{\sigma_k}+s(y_k(t-\beta z_k(t))-y_{\sigma_k}),\,\sigma_k)
).\end{align*}
Put
\begin{align}{\label{eq-3-35}}
v_k(t)=\frac{y_k(t)-y_{\sigma_k}}{\|y_k-y_{\sigma_k}\|}.
\end{align}
Then we have
\begin{align}{\label{eq-3-37}}
v_k(t-\beta_k z_k(t))=\frac{y_k(t-\beta_k
z_k(t))-y_{\sigma_k}}{\|y_k -y_{\sigma_k}\|}.
\end{align}
 By  (\ref{eq-3-34})
and (\ref{eq-3-37}), we have
\begin{align}\label{eq-3-38}
\dot{v}_k(t)=&\frac{\alpha_k}{\beta_k}\int_0^1 
\partial_1 f_k(\sigma_k,\,s) (t) \mathrm{d}s\, v_k(t)\notag\\
&+\frac{\alpha_k}{\beta_k}\int_0^1 \partial_2 f_k(\sigma_k,\,s) (t)\mathrm{d}s\,v_k(t-\beta_k z_{\sigma_k})\notag\\
&+ \frac{1-\alpha_k}{\beta_k}\int_0^1 \partial_1 f_k(\sigma_k,\,s)(t)  \mathrm{d}s \,v_k(t)\notag\\
&+\frac{1-\alpha_k}{\beta_k} \int_0^1 \partial_2 f_k(\sigma_k,\,s) (t)\mathrm{d}s\, v_k(t-\beta_k z_{\sigma_k}).
\end{align}
We show that there exists a convergent subsequence of
$\{v_k\}_{k=1}^{+\infty}$. Indeed, by (\ref{eq-3-29}) we know that $\{z_k,\,\beta_k\}_{k=1}^{+\infty}$ is uniformly bounded in $C(\mathbb{R};\mathbb{R})\times\mathbb{R}$.  Therefore we have
% $
% \|\partial_1
% g(\sigma_k)(y_k-y_{\sigma_k})+\partial_2g(\sigma_k)(z_k-z_{\sigma_k})\|\leq 1
% $. Then  it follows from (S2) that
% \begin{align*}
% |\dot{z_k}(t)|=&\left|\frac{\alpha_k}{\beta_k}\,g(y_k(t),\,z_k(t),\,\sigma_k)\right.\\
%  & \left.+\frac{1-\alpha_k}{\beta_k}(\partial_1 g(\sigma_k)(y_k-y_{\sigma_k})
%  +\partial_2g(\sigma_k)(z_k-z_{\sigma_k})\right|\notag\\
% < &\frac{\alpha_k}{\beta_k}+\frac{1-\alpha_k}{\beta_k}\\
% = &\frac{1}{\beta_k},
% \end{align*}
% which implies that
\begin{align}\label{eq-3-36}
\lim_{t\rightarrow+\infty} t-\beta_k z_k(t)=+\infty.
\end{align}
By (\ref{eq-3-35}),
(\ref{eq-3-37}) and (\ref{eq-3-36}), we have $\|v_k\|=1,\,\|v_k(\cdot-\beta_k
z_k)\|=1$. Note that by (S1) and (\ref{eq-3-31}) and by an argument similar yielding (\ref{uniform-convergence-a}), we know that
\begin{align}\label{uniform-convergence-b} 
\lim_{k\rightarrow+\infty}\|\partial_1 f_k(\sigma_k,\,s)-\partial_1
f(\sigma^*)\|=0,\,\,
\lim_{k\rightarrow+\infty}\|\partial_2 f_k(\sigma_k,\,s)-\partial_2
f(\sigma^*)\|=0, 
\end{align}
uniformly for $s\in [0, 1]$. We know from (\ref{uniform-convergence-b}) that
$\|\partial_1
f_k(\sigma_k,\,s)\|$, $\|\partial_2 f_k(\sigma_k,\,s)\|$,
 are both uniformly bounded for every
$k\in\mathbb{N}$ and $s\in [0,\,1]$. It follows from
(\ref{eq-3-38}) that there exists $\tilde{L}_3>0$ such that $\|\dot{v}_k\|<\tilde{L}_3$
for every $k\in\mathbb{N}$. By the Arzela-Ascoli Theorem, there exists a convergent
subsequence $\{v_{k_j}\}_{j=1}^{+\infty}$ of
$\{v_k\}_{k=1}^{+\infty}$. That is, there exists $v^*\in\{v\in V:
\|v\|=1\}$ such that
\begin{align}{\label{eq-3-39}}
\lim_{j\rightarrow+\infty}\|v_{k_j}-v^*\|=0.
\end{align}
By the Integral Mean Value Theorem, we obtain for all $t\in\mathbb{R}$, 
\begin{align}
&|v_{k_j}(t-\beta_{k_j} z_{k_j}(t))-v_{k_j}(t-\beta^* z_{\sigma^*})|\notag\\
= & \left|\int_0^1 \dot{v}_{k_j}(t-\beta^*
z_{\sigma^*}-\theta(\beta_{k_j} z_{k_j}(t)-\beta^*
z_{\sigma^*}))\mathrm{d}\theta(\beta_{k_j} z_{k_j}(t)-\beta^*
z_{\sigma^*})\right|\notag\\
\leq & \|\dot{v}_{k_j}\|\cdot |\beta_{k_j} z_{k_j}(t)-\beta^*
z_{\sigma^*}|\notag\\
\leq & \tilde{L}_3 (\beta_{k_j}
|z_{k_j}(t)-z_{\sigma^*}|+|\beta_{k_j}-\beta^*|z_{\sigma^*}).\label{eq-3-40}
\end{align}
Then by (\ref{eq-3-31}) and (\ref{eq-3-40}) we have
\begin{align}\label{eq-3-41}
\lim_{j\rightarrow+\infty}\|v_{k_j}(\cdot-\beta_{k_j} z_{k_j}(\cdot))-v_{k_j}(\cdot-\beta^*
z_{\sigma^*})\|=0.
\end{align}
From (\ref{eq-3-39}) and  (\ref{eq-3-41}) we have
\begin{align}\label{eq-3-42}
\lim_{j\rightarrow+\infty}\|v_{k_j}(\cdot-\beta_{k_j} z_{k_j}(\cdot))-v^*(\cdot-\beta^*
z_{\sigma^*})\|=0.
\end{align}
It follows from (\ref{eq-3-31}), (\ref{uniform-convergence-b}), (\ref{eq-3-39}) and (\ref{eq-3-42})
that the right hand side of (\ref{eq-3-38}) converges uniformly to
the right hand side of (\ref{eq-3-32}). Therefore,
\begin{align}\label{eq-3-43}
\lim_{j\rightarrow+\infty}|\dot{v}_{k_j}(t)-\dot{v}^*(t)|=0,
\end{align}
and
\begin{align}
\dot{v}^*(t)=\frac{1}{\beta^*} \partial_1 f(\sigma^*)  v^*(t)+\frac{1}{\beta^*} \partial_2 f(\sigma^*)  v^*(t-\beta^*\tau_{\sigma^*}).\label{eq-3-44}
\end{align}
Noticing that $v^*\in \{v: \|v\|=1\}$, we have $\|v^*\|\neq 0$. Since the matrix
 $\partial_1 f(\sigma^*)+\partial_2 f(\sigma^*)$ is nonsingular, $v^*$ is a nonconstant periodic
solution of (\ref{eq-3-44}). 
%Therefore,
%$(v^*,\,\sigma^*,\,\beta^*)$ is a nonconstant
%periodic solution of the linear equation (\ref{eq-3-44}). 
Then by Lemma~\ref{linear-eq-center}
$(y_{\sigma^*},\,\sigma^*,\,\beta^*)$ is also a center of
(\ref{eq-3-5})  in $\overline{B(y_0,\,\sigma_0,\,\beta_0;
r,\,\rho)}$. This contradicts  the assumption that
$B(y_0,\,\sigma_0,\,\beta_0; r,\,\rho)$ is a special neighborhood
of (\ref{eq-3-5}) which contains only one center $(y_0,\,\sigma_0,\,\beta_0)$. This completes the proof.\hfill{ }\qed
\end{proof}
Now we  are in the position to  prove a local Hopf bifurcation
theorem for system~(\ref{SDDE-general}).
\begin{theorem}\label{localhopf}
Assume $($S1$\,)$--$($S3$\,)$ hold. Let $(x_{\sigma_0},\, \sigma_0)$ be
an isolated center of system
$($\ref{eq-3-1}$\,)$.   If  the crossing number defined by $($\ref{crossingnumber}$\,)$ satisfies
\[\gamma(x_{\sigma_0},\, \sigma_0, \,\beta_0)\neq
0,\] then there exists a bifurcation of nonconstant periodic
solutions of $($\ref{SDDE-general}$\,)$ near
$(x_{\sigma_0}, \,\sigma_0)$. More precisely, there exists a
sequence  $\{(x_n, \,\sigma_n,
\beta_n)\}$ such that $\sigma_n\rightarrow\sigma_0$,
$\beta_n\rightarrow\beta_0$ as $n\rightarrow\infty$,  and
$\lim_{n\rightarrow\infty}\|x_n-x_{\sigma_0}\|=0$, where
\[
(x_n,\, \sigma_n)\in C(\mathbb{R};\mathbb{R}^{N})\times\mathbb{R}
\] is a nonconstant $2\pi/\beta_n$-periodic solution of system  $($\ref{SDDE-general}$\,)$.
\end{theorem}
\begin{proof}  
%We first normalize the period of the periodic solutions of
%system (\ref{SDDE-general}) to $2\pi$ 
Let $(x,\,\tau)$ be a solution of system $($\ref{SDDE-general}$\,)$ with $x$ being $2\pi/\beta$-periodic and $\beta>0$. Let $(x(t),\tau(t))=(y(\beta t),\,z(\beta t))$. Then
system $($\ref{SDDE-general}$\,)$ is transformed to
\begin{align}\label{eq-3-22-b}
\left\{
\begin{array}{ll}
\dot{y}(t)&=\frac{1}{\beta}f(y(t),\,y(t-\beta z(t)),\,\sigma),\\
z(t)&= g(y(t),\,y(t-\beta z_\sigma),\,\sigma).\\
\end{array}
\right.
\end{align}
Then $x$ is a ${2\pi}/{\beta}$-periodic solution of system $($\ref{SDDE-general}$\,)$  
if and only if $y$ is a $2\pi$-periodic solution of system  (\ref{eq-3-22-b}). 

Let 
$V=C_{2\pi}(\mathbb{R};\,\mathbb{R}^{N})$. For any $\xi=e^{i\nu}\in S^1$, $u\in V$, $(\xi u)(t):=u(t+\nu)$.  Recall that $\delta$ and $\varepsilon$ are defined before (\ref{crossingnumber}). Let
$\mathscr{D}(\sigma_0,\,\beta_0)=(\sigma_0-\delta,\,\sigma_0+\delta)\times
(\beta_0-\varepsilon,\,\beta_0+\varepsilon)$ and define the maps
\begin{align*} 
\begin{aligned}
 L_0y(t): & = \dot{y}(t),\, y\in C_{2\pi}^1(\mathbb{R};\,\mathbb{R}^{N}),\\
 N_0(y, \,\sigma,\,\beta)(t):& = f(y(t),\,y(t-\beta z(t)),\sigma), y\in V,\\
 \tilde{N}_0(y, \,\sigma,\,\beta)(t): & = \partial_1 f(\sigma)(y(t)-y_{\sigma})+\partial_2 f(\sigma)(y(t-\beta z_{\sigma})-y_{\sigma}), y\in V,
\end{aligned}
\end{align*}
where 
$(\sigma,\,\beta)\in\mathscr{D}(\sigma_0,\,\beta_0)$ and
$t\in\mathbb{R}$, and $(y_{\sigma},\,z_{\sigma})$ is
the stationary state of the system at $\sigma$ such that $y_{\sigma_0}=x_{\sigma_0}$. The space $V$ is a Banach
representation of the group $G=S^1$.
 
Define the operator $K: V\rightarrow
\mathbb{R}^{N}$ by
\begin{align*} 
K(y):=\frac{1}{2\pi}\int_0^{2\pi}y(t) dt,\,y\in V.
\end{align*}
By Lemma~\ref{bounded-inverse}, the operator $L_0+K: C_{2\pi}^1(\mathbb{R};\mathbb{R}^{N})\rightarrow V$ has a compact inverse
$(L_0+K)^{-1}: V\rightarrow V$.
Then, finding a $2\pi/\beta$-periodic solution for the system
(\ref{SDDE-general}) is equivalent to finding a solution of the following
fixed point problem:
\begin{align}
y=(L_0+K)^{-1}\left[\frac{1}{\beta}N_0(y,\,\sigma,\,\beta)+K(y)\right],
\end{align}
where $(y,\,\sigma,\,\beta)\in V\times\mathbb{R}\times(0,\,+\infty)$.

Define the following maps $\mathcal{F}:
V\times\mathbb{R}\times(0,\,+\infty)\rightarrow V$ and $\tilde{\mathcal{F}}:
V\times\mathbb{R}\times(0,\,+\infty)\rightarrow V$ by
\begin{align*}
&\mathcal{F}(y,\,\sigma,\,\beta):=y-(L_0+K)^{-1}\left[\frac{1}{\beta}N_0(y,\,\sigma,\,\beta)+K(y)\right],\\
&\widetilde{\mathcal{F}}(y,\,\sigma,\,\beta):=y-(L_0+K)^{-1} \left[\frac{1}{\beta}
\tilde{N}_0(y, \,\sigma,\,\beta)+K (y)\right].%\label{eq-3-46}
\end{align*}
 Finding a
${2\pi}/{\beta}$-periodic solution of  system (\ref{SDDE-general}) is
equivalent to finding the solution of the problem
\begin{align*}%\notag%\label{operatorequation}
\mathcal{F}(y,\,\sigma,\,\beta)=0,\quad (y, \,\sigma,\,\beta)\in
V\times\mathbb{R}\times(0,\,+\infty).
\end{align*}

The idea of the proof in the sequel
is to verify all the conditions (A1)-(A6) for applying  Theorem~2.4 on Hopf bifurcation developed in \cite{MR2644135}:
\begin{enumerate}
\item[(A1)] $V$ has an $S^1$-isotypical decomposition $V=\overline{\oplus_{k=0}^\infty V_k}$ and for each integer $k=0,\,1,\,2\cdots,$ the subspace $V_k$ is of finite dimension.
\item[(A2)]  There exists a compact resolvent $K$ of $L_0$  such that for every fixed parameter $(\sigma,\,\beta)\in\mathbb{R}^2$,
$(L_0+K)^{-1}\circ[N_0(\cdot,\,\sigma,\,\beta)+K]: V\rightarrow V$ is a condensing map.
\item[(A3)] There exists a 2-dimensional submanifold $M\subset V_0\times\mathbb{R}^2$ such that i) $M\subset\mathscr{F}^{-1}(0)$; ii) if $(y_0,\,\sigma_0,\,\beta_0)\in M$, then there exists an open neighborhood $U_{(\sigma_0,\,\beta_0)}$ of $(\sigma,\,\beta)$ in $\mathbb{R}^2$ , an open neighborhood $U_{y_0}$ of $U_0$ in $V_0$, and a $C^1$-map $\eta: U_{(\sigma_0,\,\beta_0)}\rightarrow U_{y_0}$ such that $M\cap (U_{y_0}\times U_{(\sigma_0,\,\beta_0)})=\{(\eta(\sigma,\,\beta),\,(\sigma,\,\beta)): (\sigma,\,\beta)\in U_{(\sigma_0,\,\beta_0)}\}$.
\item[(A4)] $M\subset\tilde{\mathscr{F}}^{-1}(0)$ and  for every fixed parameter $(\sigma,\,\beta)\in\mathbb{R}^2$,
$(L_0+K)^{-1}\circ[\tilde{N}_0(\cdot,\,\sigma,\,\beta)+K]: V\rightarrow V$ is a condensing map.
\item[(A5)] There exist $r>0$ and $\rho>0$ so that $B(y_0,\,\sigma_0,\,\beta_0; r,\,\rho)$ is a special neighborhood of $\widetilde{\mathcal{F}}$ and  there exists
$r'\in (0,\, r]$ such that $\mathscr{F}(y,\,\sigma,\,\beta)\neq 0$
for $(y,\,\sigma,\,\beta)\in\overline{B(y_0,\,\sigma_0,\,\beta_0;
r',\,\rho)}$ with $y\neq \eta(\sigma,\,\beta)$ and
$|(\sigma,\,\beta)-(\sigma_0,\beta_0)|=\rho$.
\item[(A6)] $D_u\widetilde{\mathcal{F}}(y_0,\,\sigma_0,\,\beta_0): V_0\rightarrow V_0$ is an isomorphism.
\end{enumerate}

By (S1) we know that the linear operator $\tilde{N}_0$ is
continuous. By Lemma~\ref{bounded-N-0}, we know that
$N_0(\cdot,\,\sigma,\,\beta): V\rightarrow V$ is continuous.
Moreover, by Lemma~\ref{bounded-inverse} the operator $(L_0+K)^{-1}:
V\rightarrow V$ is compact and hence
$(L_0+K)^{-1}\circ(\frac{1}{\beta}N_0(\cdot, \alpha, \beta )+K):
V\rightarrow V$ and
$(L_0+K)^{-1}\circ(\frac{1}{\beta}\tilde{N}_0(\cdot, \alpha,
\beta)+K): V\rightarrow V$ are completely continuous and hence are
condensing maps. That is, (A2) and (A4) are satisfied.

Since $(x_{\sigma_0}, \,\sigma_0)=(y_{\sigma_0},\,\sigma_0)$ is an isolated center of system (\ref{eq-3-1}) with a purely imaginary characteristic value $i\beta_0$, $\beta_0>0$, $(y_{\sigma_0},\,\sigma_0,\,\beta_0)\in V\times\mathbb{R}\times(0,\,+\infty)$ is an isolated $V$-singular point of
$\tilde{\mathcal{F}}$. That is,   $(y_{\sigma_0},\,\sigma_0,\,\beta_0)$ is the only point in $V$ such that the derivative $D_y\mathcal{F}(y_{\sigma_0},\,\sigma_0,\,\beta_0)$ is not an automorphism of $V$.  One can define the following two-dimensional
submanifold $M\subset V^G\times\mathbb{R}\times(0,\,+\infty)$ by
\begin{align*}
M:=\{(y_\sigma,\,\sigma,\,\beta): \sigma\in
(\sigma_0-\delta,\,\sigma_0+\delta),\,\beta\in
(\beta_0-\varepsilon,\,\beta_0+\varepsilon)\},
\end{align*}
such that the point $(y_{\sigma_0},\,\sigma_0,\,\beta_0)$ is the
only $V$-singular point  of $\tilde{\mathcal{F}}$ in $M$. $M$ is the set of
trivial solutions to the system (\ref{eq-3-1}) and satisfies the
assumption (A3).

Since $(y_{\sigma_0},\,\sigma_0,\,\beta_0)\in V\times\mathbb{R}\times(0,\,+\infty)$ is an isolated $V$-singular point of
$\tilde{\mathcal{F}}$, for $\rho>0$ sufficiently small, the linear operator
$D_u \tilde{\mathcal{F}}(y_\sigma,\,\sigma,\,\beta): V\rightarrow V$ with $|(\sigma,\,\beta)-(\sigma_0,\,\beta_0)|<\rho$, is not an automorphism only if $(\sigma,\,\beta)=(\sigma_0,\,\beta_0)$.
Then, by the Implicit Function Theorem, there exists $r>0$ such that for every $(y,\,\sigma,\,\beta)\in V\times\mathbb{R}\times(0,\,+\infty)$ with $|(\sigma,\,\beta)-(\sigma_0,\,\beta_0)|=\rho$ and $0<\|y-y_\sigma\|\leq r$, we have
$\tilde{\mathcal{F}}(y,\,\sigma,\,\beta)\neq 0$.  Then the set $B(x_0,\,\sigma_0,\,\beta_0; r, \rho)$ defined by
\begin{align*}
\{(y,\, \sigma,\,\beta)\in V\times\mathbb{R}\times(0,\,+\infty);
|(\sigma,\,\beta)-(\sigma_0,\,\beta_0)|<\rho, \|y-y_\sigma\|<r\},
\end{align*} is a special neighborhood for $\tilde{\mathcal{F}}$.

By Lemma~\ref{lemma-3-1}, there exists a special neighborhood
$\mathcal{U}=B(y_{\sigma_0},\,\sigma_0,\,\beta_0;r', \rho)$ such that
$\mathcal{F}$  and $\tilde{\mathcal{F}}$ are
nonzero for
$(y,\,\sigma,\,\beta)\in\overline{B(y_{\sigma_0},\,\sigma_0,\,\beta_0;r',
\rho)}$ with $y\neq y_\sigma$ and
$|(\sigma,\,\beta)-(\sigma_0,\,\beta_0)|=\rho$. That is, (A5) is satisfied.

Let $\theta$ be a completing function on $\mathcal{U}$. It
follows from Lemma~\ref{homotopy-lemma} that $
(\mathcal{F},\,\theta) $ is homotopic to
$(\tilde{\mathcal{F}},\,\theta)$ on
$\mathcal{U}$.

It is known that $V$ has the following isotypical direct sum
decomposition
\begin{align*}
%\displaystyle
V=\overline{\bigoplus\limits_{k=0}^{\infty} V_k},
\end{align*}
where $V_0$ is the space of all constant mappings from $\mathbb{R}$ into $\mathbb{R}^{N}$, and   $V_k$ with $k>0$, $k\in\mathbb{N}$ is the vector space of all mappings of the form 
\[
x\cos k\cdot+y\sin k\cdot: \mathbb{R}\ni t\rightarrow x\cos k t+y\sin k t\in\mathbb{R}^{N},
\] where  $x,\,y\in \mathbb{R}^{N}$.
Then  $V_k$, $k>0,\,k\in\mathbb{N}$, are finite dimensional.  Then,  (A1) is satisfied.

For $(\sigma,\,\beta)\in\mathscr{D}(\sigma_0,\,\beta_0)$, we denote by $\Psi(\sigma,\,\beta)$ the map $D_y\widetilde{\mathcal{F}}(y(\sigma),\,\sigma,\,\beta): V\rightarrow  V$.  Then we have $\Psi(\sigma,\,\beta)(V_k)\subset V_k$ for all
$k=0,\,1,\,2,\,\cdots$.  Therefore, we can define $\Psi_k:
\mathscr{D}(\sigma_0,\,\beta_0)\rightarrow L(V_k,\,V_k)$ by
\begin{align*}
\Psi_k(\sigma,\,\beta):=\Psi(\sigma,\,\beta)|_{V_k}.
\end{align*}
We note that  $V_k$, $k\geq 1,\,k\in\mathbb{N}$, can be endowed with the natural complex structure $J: V_k\rightarrow V_k$ defined by
 \[
 J(x\cos k\cdot+y\sin k\cdot)=-x\sin k\cdot+y\cos k\cdot),\,x\,\,y\in\mathbb{R}^N.
\] By extending the linearity of $J$ to the vector space  spanned over the field of complex numbers  by $e^{ik\cdot}\cdot\epsilon_j:\mathbb{R}\ni t\rightarrow e^{ik t}\cdot\epsilon_j\in\mathbb{C}^{N},\,j=1,\,2,\,\cdots,\,N$, we know that
\[
\{e^{ik\cdot}\cdot\epsilon_j,\,J(e^{ik\cdot}\cdot\epsilon_j)\}_{j=1}^{N}=\{e^{ik\cdot}\cdot\epsilon_j,\,ie^{ik\cdot}\cdot\epsilon_j\}_{j=1}^{N}
\] is a basis of $V_k$,  where
$\{\epsilon_1,\,\epsilon_2,\,\cdots,\epsilon_{N}\}$ denotes the
standard basis of $\mathbb{R}^{N}$.    Then we identify $V_k$ with the vector space over the complex numbers  spanned by
$e^{ik\cdot}\cdot\epsilon_j,\,j=1,\,2,\,\cdots,\,N$. 

Then we have for $v_k\in V_k$, $k\in\mathbb{Z}$, $k\geq 1$,
\begin{align*}
 {\Psi}_k(\sigma, \beta) v_k &=
v_k-(L_0+K)^{-1}\left(\frac{1}{\beta}D_u \tilde{N}_0(u(\sigma),\,\sigma,\,\beta)+K\right)v_k\\
&=v_k-\frac{1}{\beta}(L_0+K)^{-1} \left(\partial_1 f(\sigma)  v_k+ \partial_2 f(\sigma) (v_k)_{\beta z_{\sigma}}\right),
\end{align*}
where $(v_k)_{\beta z_{\sigma}}=v_k(\cdot-\beta z_{\sigma})$. Then we have, for $e^{ik\cdot}\epsilon_j\in V_k$,
\begin{align*}
&\, {\Psi}_k(\sigma,\,\beta)(e^{ik\cdot}\epsilon_j)\\ 
=&\frac{1}{ik\beta} \left( ik\beta\,\mathrm{Id}-\partial_1 f(\sigma)-\partial_2 f(\sigma)e^{-ik\beta z_\sigma}  \right)\cdot(e^{ik\cdot}\epsilon_j)\\
=&\frac{1}{ik\beta}  
\Delta_{(u(\sigma),\,\sigma)}(ik\beta) \cdot (e^{ik\cdot}\epsilon_j),
\end{align*}
where the last equality follows from (\ref{character-matrix}).
Therefore, the matrix representation $[{\Psi}_k]$ of ${\Psi}_k(\sigma,\,\beta)$
with respect to the ordered $\mathbb{C}$-basis 
$\{e^{ik\cdot}\epsilon_j\}_{j=1}^{N}$
is given by
\begin{align*}
\frac{1}{ik\beta}  
\Delta_{(y_\sigma,\,\sigma)}(ik\beta).
\end{align*}

Next we  show that there exists
some $k\in\mathbb{Z}$, $k\geq 1$, such that
$\mu_k(y_{\sigma_0},\,\sigma_0,\beta_0):=\deg_B(\det_\mathbb{C}[\Psi_k])\neq 0$.

Define $\Psi_H:
\mathscr{D}(\sigma_0,\,\beta_0)\rightarrow\mathbb{R}^2\simeq
\mathbb{C}$ by
\begin{align*}
\Psi_H(\sigma,\,\beta)%&=\det\left[
&=\det\Delta_{(y_\sigma,\, \sigma)}(i\beta).
\end{align*}
The number
$\mu_1(y_{\sigma_0},\,\sigma_0,\,\beta_0)$   can be written as follows (see Theorem~7.1.5 of \cite{kw}):
\begin{align*}
\mu_1(u(\sigma_0),\,\sigma_0,\,\beta_0)
%&=\epsilon\cdot \deg\left(\det\Delta_{(u(\sigma),\,\sigma)}(i\beta), \mathscr{D}(\sigma_0,\,\beta_0)\right)\\
%&=\epsilon\cdot \deg\left(
%\det\Delta_{(u(\sigma),\,\sigma)}(i\beta) , \mathscr{D}(\sigma_0,\,\beta_0)\right)\\
&=\epsilon\cdot \deg\left( \Psi_H,
\mathscr{D}(\sigma_0,\,\beta_0)\right),
\end{align*}
where $\epsilon=\mathrm{sign}\det \Psi_0(\sigma,\,\beta)$ for
$(\sigma,\beta)\in\mathscr{D}(\sigma_0,\,\beta_0)$. For a constant map $v_0\in V_0$,
\[
\Psi_0(\sigma,\,\beta)v_0=-\frac{1}{\beta} (\partial_1 f(\sigma)+\partial_2 f(\sigma))v_0.
\]Then, by (S3), we have
$\epsilon\neq 0$ and therefore (A6) is satisfied.

Note that $\alpha_0,\,\beta_0,\,\delta$ and $\varepsilon$ are chosen at (\ref{crossingnumber}). Define the function
$H:
[\sigma_0-\delta,\,\sigma_0+\delta]\times\overline{\Omega}\rightarrow\mathbb{R}^2\simeq\mathbb{C}$
by
\begin{align*}
H(\sigma,\,\alpha,\,\beta):=\det
\Delta_{(y_\sigma,\,\sigma)}(\alpha+i\beta),
\end{align*}
where
$\Omega=(0,\,\alpha_0)\times(\beta_0-\varepsilon,\,\beta_0+\varepsilon)$,
$\alpha_0=\alpha_0(\sigma_0,\,\beta_0)>0$. By the same argument for (\ref{character-pure}) and (\ref{crossingnumber}), we know that $H$ satisfies all the
conditions of Lemma~2.1 of \cite{MR2644135} (or Lemma 7.2.1 of \cite{kw}) by the choice of
$\alpha_0,\,\beta_0,\,\varepsilon$ and $\delta$. So we have
\begin{align*}
\deg\left( \Psi_H,
\mathscr{D}(\sigma_0,\,\beta_0)\right)=\gamma(y_{\sigma_0},\,\sigma_0,\,\beta_0)\neq
0.
\end{align*}
Thus, $\mu_1(y_{\sigma_0},\,\sigma_0,\,\beta_0)\neq 0$ which, by
Theorem~2.4 of  \cite{MR2644135}, implies that $
(y_{\sigma_0},\,\sigma_0,\,\beta_0)$ is a bifurcation point of the
system (\ref{eq-3-22-b}). Consequently, there exists a sequence of
non-constant periodic solutions
$(x_n,\,\sigma_n,\,\beta_n)$
such that $\sigma_n\rightarrow \sigma_0$, $\beta_n\rightarrow
\beta_0$ as $n\rightarrow\infty$, and   $x_n$ is a
$2\pi/\beta_n$-periodic solution of (\ref{SDDE-general}) such that the associated pair $(x_n,\,\tau_n)$ with $\tau_n(t)=g(x_n(t),\,x_n(t-\tau_n(t)),\,\sigma_n)$ satisfies (\ref{SDDE-general}) with
$\lim_{n\rightarrow+\infty}\|(x_n,\,\tau_n)-(x_{\sigma_0},\,\tau_{\sigma_0})\|=0$.
\hfill{ }\qed \end{proof}

\section{Local Hopf bifurcation for system~(\ref{SDDE-general-original})}\label{local-bifurcations-new}

Now we consider the local Hopf bifurcation problem of   system~(\ref{SDDE-general-original}). By Lemma~\ref{Lemma-2-1}, we know that if $x\in C^1(\mathbb{R};\mathbb{R}^N)$ is $p$-periodic and is in a small neighborhood  $O(\epsilon)$ of $x_\sigma$, there exists a unique $p$-periodic $\tau\in C^1(\mathbb{R};\mathbb{R})$ such that $\tau(t)=g(x(t),\,x(t-\tau(t)),\,\sigma),\,t\in\mathbb{R}$. We call $x$ a solution if   $(x,\,\tau)$ satisfies system~(\ref{SDDE-general}).

For a stationary state $x_0$ of system~(\ref{SDDE-general-original}) with the
parameter $\sigma _0$, we say that $(x_0,\,\sigma_0)$ is a Hopf
bifurcation point of system (\ref{SDDE-general-original}), if there exist a sequence
$\{(x_k,\,\sigma_k,\,T_k)\}_{k=1}^{+\infty}\subseteq C^1(\mathbb{R};
\mathbb{R}^{N})\times\mathbb{R}^2$ and $T_0> 0$ such that
$$
\lim_{k\rightarrow+\infty}\|(x_k,\,\sigma_k,\,
T_k)-(x_0,\,\sigma_0,\,T_0)\|_{C^1(\mathbb{R};
\mathbb{R}^{N})\times\mathbb{R}^2}=0,
$$
and $(x_k,\,\sigma_k)$ is a nonconstant $T_k$-periodic solution of
system (\ref{SDDE-general-original}). 

We freeze the state-dependent delay in system~(\ref{SDDE-general-original}) at its stationary state  and linearize the resulting differential equation of $x$ with constant delay at the stationary state.
For $\sigma\in
(\sigma_0-\epsilon_0,\,\sigma_0+\epsilon_0)$,  the following formal
linearization of system (\ref{SDDE-general-original}) at the stationary point
$x_{\sigma}$:
\begin{align}\label{new-eq-3-1-original}
\dot{x}(t) =  
\partial_1 f(\sigma)  \left( 
 x(t)-x_\sigma 
\right) + 
\partial_2 f(\sigma) \left(  x(t-\tau_\sigma)-x_\sigma
\right),
\end{align}
where
\begin{align*}
\partial_1 f(\sigma)&:=\partial_1 f(x_\sigma,
\,\tau_\sigma,\,\sigma), \,
\partial_2 f(\sigma):=\partial_2 f(x_\sigma,
\,\tau_\sigma,\,\sigma),\,\tau_\sigma=g(x_\sigma,\,x_\sigma,\,\sigma).
\end{align*}Notice that the system~(\ref{new-eq-3-1-original}) is the same as system~(\ref{eq-3-1}) and hence they share the same characteristic equations.

Let $(x_{\sigma_0}, \,\sigma_0)$ be
an isolated center of (\ref{eq-3-1}) and let $O(\epsilon_0)\subset C^1(\mathbb{R};\mathbb{R}^N)$ be a neighborhood of $x_{\sigma_0}$.  In the following we confine the discussion with $x\in O(\epsilon_0)\subset C^1(\mathbb{R};\mathbb{R}^N)$, where by Lemma~\ref{Lemma-2-1}, $\epsilon_0>0$ is chosen so that every $p$-periodic $x\in O(\epsilon_0)\subset C^1(\mathbb{R};\mathbb{R}^N)$ determines a unique continuously differentiable $p$-periodic $\tau$.

Now we formulate the Hopf bifurcation problem as a fixed point problem
in $C^1(\mathbb{R};\mathbb{R}^N)$. We normalize
the period of the $2\pi/\beta$-periodic solution $x\in O(\epsilon_0)$ of (\ref{SDDE-general}) and the associated $\tau\in C^1(\mathbb{R};\mathbb{R})$ by setting
$(x(t),\tau(t))=(y(\beta t),\,z(\beta t))$. We obtain
\begin{align}\label{new-eq-3-4}
\left(
\begin{array}{c}
\dot{y}(t)\\
 {z}(t)
\end{array}
\right)=\left(
\begin{array}{c}
\frac{1}{\beta}f(y(t),\,y(t-\beta z(t)),\,\sigma)\\
g(y(t),\,y(t-\beta z(t)),\,\sigma)
\end{array}
\right).
\end{align}

%In the following presentation, we drop the algebraic equation of system~(\ref{eq-3-4}) and  call 
%$y\in C_{2\pi}(\mathbb{R};\mathbb{R}^N)$ a solution.
Let $W=O(\epsilon_0)\cap C^1_{2\pi}(\mathbb{R};\mathbb{R}^N)$. Define $N_1: W\ni(y,\,\sigma,\,\beta) \times\mathbb{R}^2\rightarrow N_1(y,\,\sigma,\,\beta)\in C^1_{2\pi}(\mathbb{R};\mathbb{R}^N)$   by
\begin{align}\label{N-1-definition}
 N_1(y,\,\sigma,\,\beta)(t)= f(y(t),\,y(t-\beta z(t)),\,\sigma). 
\end{align}Then the equation for $\dot{y}$ in system~(\ref{new-eq-3-4}) is rewritten as
\begin{align}\label{new-eq-3-4-new}
\dot{y}(t) = \frac{1}{\beta}N_1(y,\,\sigma,\,\beta)(t). 
\end{align}
Correspondingly, (\ref{new-eq-3-1-original}) is transformed into
\begin{align}\label{new-eq-3-5}
\dot{y}(t) =\frac{1}{\beta}\tilde{N}_1(y,\,\sigma,\,\beta)(t),
\end{align}
where $\tilde{N}_1:  W\ni(y,\,\sigma,\,\beta) \times\mathbb{R}^2\rightarrow \tilde{N}_1(y,\,\sigma,\,\beta)\in C^1_{2\pi}(\mathbb{R};\mathbb{R}^N)$ is defined by
\begin{align*}
\tilde{N}_1(y,\,\sigma,\,\beta)(t)=  \partial_1 f(\sigma)   \left(y(t)-y_\sigma\right)
+  
\partial_2 f(\sigma) \left(y(t-\beta z_\sigma)-y_\sigma\right).
\end{align*} with $(y_{\sigma},\,z_{\sigma})=(x_{\sigma},\,\tau_{\sigma})$.
We note that $y$ is $2\pi$-periodic if and only if $x$
is $(2\pi/\beta)$-periodic. 
 
%$S^1\backsimeq\mathbb{R}/2\pi\mathbb{Z}$ and $V=C(S^1;\,\mathbb{R}^{N+1})\backsimeq C(\mathbb{R}/2\pi\mathbb{Z};\,\mathbb{R}^{N+1})$. That is, we identify a %2$\pi$-periodic function $u:\mathbb{R}\rightarrow\mathbb{R}^{N+1}$ with the  function $\tilde{u}: S^1\rightarrow\mathbb{R}^{N+1}$ via $u(t)=\tilde{u}(e^{it})$.

Let $L_0: C_{2\pi}^1(\mathbb{R};\,\mathbb{R}^{N})\rightarrow C_{2\pi}(\mathbb{R};\,\mathbb{R}^{N})$ be defined by $L_0
y(t)=\dot{y}(t), t\in \mathbb{R}$
and $K: C^1_{2\pi}(\mathbb{R};\mathbb{R}^N)\rightarrow \mathbb{R}^{N}$
 be defined by 
 \begin{align}\label{new-K-map}
 K(y)=\frac{1}{2\pi}\int_0^{2\pi}y(t) \mathrm{d} t.
 \end{align}

Define the map $\tilde{\mathcal{F}}:
W\times\mathbb{R}^2\rightarrow C^1_{2\pi}(\mathbb{R};\mathbb{R}^N)$ by
\begin{align}\label{tilde-F}
&\widetilde{\mathcal{F}}(y,\,\sigma,\,\beta):=y-(L_0+K)^{-1}
\left[\frac{1}{\beta} \tilde{N}_0(y, \,\sigma,\,\beta)+K (y)\right].
\end{align} We suppose that the set defined by
\[
B(y_0,\,\sigma_0,\,\beta_0; r,\,\rho)=\{(y,\,\sigma,\,\beta):
\|y-y_\sigma\|_{C^1}<r,\,|(\sigma,\,\beta)-(\sigma_0,\,\beta_0)|<\rho\},
 \]  is a  special neighborhood  of $\widetilde{\mathcal{F}}$
  which satisfies 
  \begin{enumerate}
  \item[i)] 
   $\widetilde{\mathcal{F}}(y,\,\sigma,\,\beta)\neq 0$ for every $(y,\,\sigma,\,\beta)\in \overline{B(y_0,\,\sigma_0,\,\beta_0; r,\,\rho)}$ with $|(\sigma,\,\beta)-(\sigma_0,\,\beta_0)|=\rho$ and $\|y-y_\sigma\|_{C^1}\neq 0$;
   \item[ii)] $(y_0,\,\sigma_0,\,\beta_0)$ is the only isolated center in 
$\overline{B(y_0,\,\sigma_0,\,\beta_0; r,\,\rho)}$. 
\end{enumerate}
Before we state and prove our local Hopf bifurcation theorem, we
need the following technical  Lemmas.
\begin{lemma}\label{new-bounded-inverse}
Let $L_0: C_{2\pi}^1(\mathbb{R};\,\mathbb{R}^{N})\rightarrow C_{2\pi}(\mathbb{R};\,\mathbb{R}^{N})$ be defined by $L_0
y(t)=\dot{y}(t), t\in \mathbb{R}$ and let $K: C_{2\pi}^1(\mathbb{R};\,\mathbb{R}^{N})\rightarrow \mathbb{R}^{N}$
 be defined at (\ref{new-K-map}). Then the inverse  $(L_0+K)^{-1}: C_{2\pi}(\mathbb{R};\,\mathbb{R}^{N})\rightarrow C_{2\pi}^1(\mathbb{R};\,\mathbb{R}^{N})$ exists and is continuous. 
\end{lemma} 
\begin{proof} By the proof of Lemma 3.1 in \cite{MR2644135}, $L_0+K:  C_{2\pi}^1(\mathbb{R};\,\mathbb{R}^{N})\rightarrow C_{2\pi}(\mathbb{R};\,\mathbb{R}^{N})$ is one-to-one and onto. Moreover,  
  $(L_0+K)^{-1}: C_{2\pi}(\mathbb{R};\,\mathbb{R}^{N})\rightarrow C_{2\pi}^1(\mathbb{R};\,\mathbb{R}^{N})$ is continuous.\qed
\end{proof}

\begin{lemma}\label{new-new-bounded-N-0}
For any $\sigma\in\mathbb{R}$ and $\beta>0$, the map $N_1(\cdot,\,\sigma,\,\beta): W\rightarrow C_{2\pi}^1(\mathbb{R};\,\mathbb{R}^{N})$ defined by (\ref{N-1-definition}) is continuous.
\end{lemma} 
\begin{proof} Let $\{y_n\}_{n=1}^\infty\subset W$ be a convergent sequence with limit $y_0\in W$.  By Lemma~\ref{Lemma-2-1}, $\{y_n\}_{n=1}^\infty\subset W$ uniquely determines a sequence $\{z_n\}_{n=1}^\infty\subset C_{2\pi}^1(\mathbb{R};\,\mathbb{R})$ satisfying
\[
\tau_n(t)=g(y_n(t),\,y_n(t-\tau_n(t)),\,\sigma),\,t\in\mathbb{R}.
\]
Moreover, there exists $z_0\in C_{2\pi}^1(\mathbb{R};\,\mathbb{R})$ such that
$\tau_0(t)=g(y_0(t),\,y_0(t-\tau_0(t)),\,\sigma),\,t\in\mathbb{R}$ and
 \begin{align}\label{zn-z0}
\lim_{n\rightarrow\infty}\sup_{t\in[0,\,2\pi]}|z_n(t)-z_0(t)|=0.
\end{align}
By taking derivatives on both sides of $z_n(t)=g(y_n(t),\,y(t-\beta z_n(t),\,\sigma)$ we obtain that
\begin{align}\label{dot-zn-eq}
\dot{z}_n(t)=\frac{\partial_1 g(y_n(t),\,y_n(t-\beta z_n(t)),\,\sigma)\dot{y}_n(t)}{1+\beta\partial_2 g(y_n(t),\,y_n(t-\beta z_n(t)),\,\sigma)\dot{y}_n(t-\beta z_n(t))}.
\end{align}
Notice that, we have
\begin{align}\label{dot-z0-eq}
\dot{z}_0(t)=\frac{\partial_1 g(y_0(t),\,y_0(t-\beta z_0(t)),\,\sigma)\dot{y}_0(t)}{1+\beta\partial_2 g(y_0(t),\,y_0(t-\beta z_0(t)),\,\sigma)\dot{y}_0(t-\beta z_0(t))}.
\end{align}Using the Triangle Inequality and the Integral Mean Value Theorem, and noticing from (S2) that $g$ is $C^2$, we can show that
\begin{align}\label{dot-zn-z0}
\sup_{t\in[0,\,2\pi]}&\left|\frac{\partial_1 g(y_n(t),\,y_n(t-\beta z_n(t)),\,\sigma)\dot{y}_n(t)}{1+\beta\partial_2 g(y_n(t),\,y_n(t-\beta z_n(t)),\,\sigma)\dot{y}_n(t-\beta z_n(t))}\right.\notag\\
&-\left.
\frac{\partial_1 g(y_0(t),\,y_0(t-\beta z_0(t)),\,\sigma)\dot{y}_0(t)}{1+\beta\partial_2 g(y_0(t),\,y_0(t-\beta z_0(t)),\,\sigma)\dot{y}_0(t-\beta z_0(t))}
\right|\notag\\
&\rightarrow 0, \,{\mbox{as }\, n\rightarrow\infty.}
\end{align}Therefore, by (\ref{dot-zn-eq}) and (\ref{dot-z0-eq}) we have $\lim_{n\rightarrow\infty}\sup_{t\in[0,\,2\pi]}|\dot{z}_n(t)-\dot{z}_0(t)|=0$ which combined with (\ref{zn-z0}) leads to $\lim_{n\rightarrow\infty}\|z_n- z_0\|_{C^1}=0.$

%With minor modifications on the proof for Lemma~3.2 in  \cite{MR2644135}, 

Next we show that $N_1: W\rightarrow C_{2\pi}^1(\mathbb{R};\,\mathbb{R}^{N})$ defined by
$
N_1(y,\,\sigma,\,\beta)(t)=f(y(t),\,y(t-\beta z(t)),\,\sigma)
$ is continuous.  That is, 
\begin{align}\label{N-1-continuity}
\lim_{n\rightarrow\infty}\left\|N_1(y_n(t),\,y_n(t-\beta z_n(t)),\,\sigma)-N_1(y_0(t),\,y_0(t-\beta z_0(t)),\,\sigma)\right\|_{C^1}=0.
\end{align}
By the proof of Lemma 3.2 in \cite{MR2644135}, we know that the restriction $N_1\vline_{\,C_{2\pi}(\mathbb{R};\,\mathbb{R}^{N})}$ is a continuous map from $C_{2\pi}(\mathbb{R};\,\mathbb{R}^{N})$ to $C_{2\pi}(\mathbb{R};\,\mathbb{R}^{N})$.  Therefore, we have
\begin{align}\label{N-1-sup}
\lim_{n\rightarrow\infty}\sup_{t\in[0,\,2\pi]}\left|f(y_n(t),\,y_n(t-\beta z_n(t)),\,\sigma)-f(y_0(t),\,y_0(t-\beta z_0(t)),\,\sigma)\right|=0.
\end{align}
Moreover, since  $\lim_{n\rightarrow\infty}\|y_n- y_0\|_{C^1}=0$
 and $\lim_{n\rightarrow\infty}\|z_n- z_0\|_{C^1}=0.$
 we can use the Triangle Inequality and the Integral Mean Value Theorem to obtain that 
 \begin{align}\label{dot-N-1-sup}
&\lim_{n\rightarrow\infty}\sup_{t\in[0,\,2\pi]}\left|\frac{\mathrm{d}}{\mathrm{d}t}N_1(y_n,\,y_n(t-\tau_n(t)),\,\sigma)-\frac{\mathrm{d}}{\mathrm{d}t}N_1(y_0,\,y_0(t-\tau_0(t)),\,\sigma)\right|\notag \\
=&\lim_{n\rightarrow\infty} \sup_{t\in[0,\,2\pi]}
\left|\partial_1f(y_n(t),\,y_n(t-\beta z_n(t)),\,\sigma)\dot{y}_n(t)+\partial_2f(y_n(t),\,y_n(t-\beta z_n(t)),\,\sigma)\right.\notag\\
& \left.\times\dot{y}_n(t-\beta z_n(t)) (1-\beta\dot{z}_n(t))-\partial_1f(y_0(t),\,y_0(t-\beta z_0(t)),\,\sigma)\dot{y}_0(t)-\partial_2f(y_0(t),\,y_0(t-\beta z_0(t)),\,\sigma)\right.\notag\\
& \left.\times\dot{y}_0(t-\beta z_0(t)) (1-\beta\dot{z}_0(t))\right|,\notag\\
& =0.
\end{align}By (\ref{N-1-sup}) and (\ref{dot-N-1-sup}), $N_1: W\rightarrow  C_{2\pi}^1(\mathbb{R};\,\mathbb{R}^{N})$ is continuous.\qed
\end{proof}

To establish the $S^1$-degree on some special neighborhood near the stationary state, we have
\begin{lemma}\label{new-lemma-3-1}
Assume $($\mbox{{S1}}$\,)$--$($\mbox{{S3}}$\,)$ hold. Let $L_0$ and $K$ be as in
Lemma~\ref{new-bounded-inverse} and $\tilde{N}_1:W\times\mathbb{R}^2\rightarrow C_{2\pi}^1(\mathbb{R};\mathbb{R}^N)$ be as in
$($\ref{new-eq-3-5}$\,)$.  Let $\tilde{\mathcal{F}}:
W\times\mathbb{R}^2\rightarrow C_{2\pi}^1(\mathbb{R};\mathbb{R}^N)$ be defined at (\ref{tilde-F}).  If $B(y_0,\,\sigma_0,\,\beta_0; r,\,\rho)$ is a special neighborhood of $\widetilde{\mathcal{F}}$    with $0<\rho<\beta_0$, then there exists
$r'\in (0,\, r]$ such that the neighborhood
\[
B(y_0,\,\sigma_0,\,\beta_0;
r',\,\rho)=\{(u,\,\sigma,\,\beta):
\|y-y_\sigma\|_{C^1}<r',\,|(\sigma,\,\beta)-(\sigma_0,\,\beta_0)|<\rho\}
\] satisfies
\[
\dot{y}(t)\not\equiv \frac{1}{\beta}  f(y(t),y(t-\beta z(t)),\,\sigma)
\]
for $(y,\,\sigma,\,\beta)\in\overline{B(y_0,\,\sigma_0,\,\beta_0;
r',\,\rho)}$ with $y\neq y_\sigma$ and
$|(\sigma,\,\beta)-(\sigma_0,\beta_0)|=\rho$.
\end{lemma}
\begin{proof}We prove by contradiction. 
Suppose the statement is not true, then for any $0<r'\leq r$, there exists
$(y,\,\sigma,\,\beta)$ such that
$0<\|y-y_\sigma\|_{C^1}<r',\,|(\sigma,\,\beta)-(\sigma_0,\,\beta_0)|=\rho$
and
\begin{align}\label{new-eq-3-13}
\dot{y}(t)=\frac{1}{\beta} 
f(y(t),y(t-\beta z(t)),\,\sigma) \textrm{ for } t\in\mathbb{R}.
\end{align}
Then there exists a sequence of nonconstant periodic solutions
$\{(y_k,\,\sigma_k,\,\beta_k)\}_{k=1}^{\infty}$
of (\ref{new-eq-3-13})
 such that
\begin{align}\label{new-eq-3-14}
&\lim_{k\rightarrow +\infty}\|y_k-y_{\sigma_k}\|_{C^1}=0,\,\,
|(\sigma_k,\,\beta_k)-(\sigma_0,\,\beta_0)|=\rho, \intertext{and}
&\dot{y}_k(t)=\frac{1}{\beta_k} 
f(y_k(t),y_k(t-\beta_k z_k(t)),\,\sigma_k) 
 \textrm{ for } t\in\mathbb{R},\label{new-eq-3-15}
\end{align}where $z_k$ is chosen according to $y_k$ in light of Lemma~\ref{Lemma-2-1} so that $(y_k,\,z_k)$ is a solution of system~(\ref{new-eq-3-4}).

Note that $0<\rho<\beta_0$ implies that $\beta_k\geq \beta_0-\rho>0$
for every $k\in\mathbb{N}$. Also, since the sequence
$\{\sigma_k,\,\beta_k\}_{k=1}^{\infty}$ belongs to a bounded
neighborhood of $(\sigma_0,\,\beta_0)$ in $\mathbb{R}^2$, there
exists a convergent subsequence, still denoted by
$\{(\sigma_k,\,\beta_k)\}_{k=1}^{\infty}$ for notational simplicity, that converges to
$(\sigma^*,\,\beta^*)$ so that
$|(\sigma^*,\,\beta^*)-(\sigma_0,\,\beta_0)|=\rho$ and $\beta^*>0$.  Then we have
\begin{align}\label{new-eq-3-16}
&\lim_{k\rightarrow +\infty}\|y_k-y_{\sigma_k}\|_{C^1}=0,\,\,
\lim_{k\rightarrow+\infty}|(\sigma_k,\,\beta_k)-(\sigma^*,\,\beta^*)|=0,
\intertext{and} &\dot{y}_k(t)=\frac{1}{\beta_k} 
f(y_k(t),y_k(t-\beta_k z_k(t)),\,\sigma_k)  \textrm{ for } t\in\mathbb{R}.\label{new-eq-3-17}
\end{align}
We need to show that the system
\begin{align}
\dot{v}(t)=\frac{1}{\beta^*} \partial_1 f(\sigma^*) v(t)+\frac{1}{\beta^*} \partial_2 f(\sigma^*)  v(t-\beta^* z_{\sigma^*}),\label{new-eq-3-18}
\end{align}
has a nonconstant periodic solution which contradicts the
assumption that $(y_{\sigma_0},\,\sigma_0,\,\beta_0)$ is the only
center of (\ref{new-eq-3-5}) in $\overline{B(u_0,\,\sigma_0,\,\beta_0;
r,\,\rho)}$. But (\ref{new-eq-3-16}) implies that
 \begin{align}\label{new-eq-3-16-b}
&\lim_{k\rightarrow +\infty}\|y_k-y_{\sigma_k}\|=0,\,\,
\lim_{k\rightarrow+\infty}|(\sigma_k,\,\beta_k)-(\sigma^*,\,\beta^*)|=0,
\end{align} Then by the same argument in the proof of Lemma~\ref{lemma-3-1}, (\ref{new-eq-3-18}) has a nonconstant periodic solution.
 This completes the proof.\hfill{ }\qed
\end{proof}
To apply the homotopy argument of $S^1$-degree, we show the following
\begin{lemma}\label{new-homotopy-lemma}
Assume (S1)--(S3) hold. Let $L_0$, $K$, $\tilde{N}_1$, $\tilde{\mathcal{F}}$ be as in Lemma~\ref{new-lemma-3-1}  and  $N_1: W\times\mathbb{R}^2\rightarrow C_{2\pi}^1(\mathbb{R};\mathbb{R}^N)$ be as in $($\ref{new-eq-3-4}$\,)$. Define the map $\mathcal{F}:
W\times\mathbb{R}^2\rightarrow C_{2\pi}^1(\mathbb{R};\mathbb{R}^N)$  by
\begin{align*}
&\mathcal{F}(y,\,\sigma,\,\beta):=y-(L_0+K)^{-1}\left[\frac{1}{\beta}N_1(y,\,\sigma,\,\beta)+K(y)\right].
\end{align*}
If $\mathcal{U}=B(y_0,\,\sigma_0,\,\beta_0; r,\,\rho)$ $\subseteq
W\times\mathbb{R}^2$ is a special neighborhood of
$\tilde{\mathcal{F}}$  with $0<\rho<\beta_0$, then there exists
$r'\in (0,\,r]$ such that
$\mathcal{F}_{\theta}=(\mathcal{F},\,\theta)$ and
$\tilde{\mathcal{F}}_{\theta}=(\tilde{\mathcal{F}},\,\theta)$  are
homotopic on $\overline{B(y_0,\,\sigma_0,\,\beta_0; r',\,\rho)}$,
where $\theta$ is a completing function (or Ize's function) defined on
$\overline{B(y_0,\,\sigma_0,\,\beta_0; r',\,\rho)}$ which satisfies
\begin{enumerate}
\item[i)] $\theta(y_\sigma,\,\sigma,\,\beta)=-|(\sigma,\,\beta)-(\sigma_0,\,\beta_0)|$ if $(y_\sigma,\,\sigma,\,\beta)\in \bar{\mathcal{U}};$
\item[ii)] $\theta(y,\,\sigma,\,\beta)=r'$ if $\|y-y_\sigma\|_{C^1}=r'$.
\end{enumerate}
\end{lemma}
\begin{proof} Since $\mathcal{U}=B(y_0,\,\sigma_0,\,\beta_0; r,\,\rho)$ $\subseteq
W\times\mathbb{R}^2$ is a special neighborhood of
$\tilde{\mathcal{F}}$  with $0<\rho<\beta_0$, then by Lemma~\ref{new-lemma-3-1}, both $\mathcal{F}_{\theta}=(\mathcal{F},\,\theta)$ and
$\tilde{\mathcal{F}}_{\theta}=(\mathcal{F},\,\theta)$ are $\mathcal{U}$-admissible. For contradiction, suppose that the conclusion is not true. Then for any $r'\in (0,\,r]$,
$\mathcal{F}_{\theta}=(\mathcal{F},\,\theta)$ and
$\tilde{\mathcal{F}}_{\theta}=(\tilde{\mathcal{F}},\,\theta)$  are not homotopic
on $\overline{B(y_0,\,\sigma_0,\,\beta_0; r',\,\rho)}$. That is, any homotopy map between
$\mathcal{F}_{\theta}$ and $\tilde{\mathcal{F}}_{\theta}$ has a zero on the boundary of $\overline{B(y_0,\,\sigma_0,\,\beta_0; r',\,\rho)}$. In particular, the linear homotopy $h(\cdot,\,\alpha):=\alpha \mathcal{F}_{\theta}+(1-\alpha)\tilde{\mathcal{F}}_{\theta}=(\alpha \mathcal{F} +(1-\alpha)\tilde{\mathcal{F}},\theta) $ has a zero on the boundary of $\overline{B(y_0,\,\sigma_0,\,\beta_0; r',\,\rho)}$, where $\alpha\in [0,\,1]$.

Note that $\theta(y,\,\sigma,\,\beta)>0$ if
$\|y-y_\sigma\|_{C^1}=r'$.
Then, there exist
$(y,\,\sigma,\,\beta)$ and $\alpha\in [0,\,1]$ such that
$\|y-y_\sigma\|_{C^1}<r',\,|(\sigma,\,\beta)-(\sigma_0,\,\beta_0)|=\rho$
and
\begin{align}
H(y,\,\sigma,\,\beta,\,\alpha):=\alpha \mathcal{F} +(1-\alpha)\tilde{\mathcal{F}}=0.\label{new-eq-3-28}
\end{align}
Since $r'>0$ is arbitrary in the interval $(0,\,r]$, there exists a
nonconstant sequence
$\{(y_k,\,\sigma_k,\,\beta_k,\,\alpha_k)\}_{k=1}^{\infty}$ of
solutions of (\ref{new-eq-3-28}) such that
\begin{align}
&      \lim_{k\rightarrow +\infty}\|y_k-y_{\sigma_k}\|_{C^1}=0,\,\,
       |(\sigma_k,\,\beta_k)-(\sigma_0,\,\beta_0)|=\rho,\, 0
       \leq \alpha_k\leq 1,
       \label{new-eq-3-29}
\intertext{and} &      H(y_k,\,\sigma_k,\,\beta_k,\,\alpha_k)=0, \mbox{ for all } k\in\mathbb{N}.
\label{new-eq-3-30}
\end{align}
Note that $0<\rho<\beta_0$ implies that $\beta_k\geq \beta_0-\rho>0$
for every $k\in\mathbb{N}$. From (\ref{new-eq-3-29}) we know that
$\{(\sigma_k,\,\beta_k,\,\alpha_k)\}_{k=1}^{\infty}$ belongs to a
compact subset of $\mathbb{R}^3$. Therefore, there exist  a
convergent subsequence, denoted for notational simplicity by
$\{(\sigma_k,\,\beta_k,\,\alpha_k)\}_{k=1}^{\infty}$ without loss of
generality, and $(\sigma^*,\,\beta^*,\,\alpha^*)\in \mathbb{R}^3$
such that $\beta^*\geq \beta_0-\rho>0$, $\alpha^*\in [0,\,1]$ and
\begin{align}\label{new-eq-3-31}
\lim_{k\rightarrow+\infty}
|(\sigma_k,\,\beta_k,\,\alpha_k)-(\sigma^*,\,\beta^*,\,\alpha^*)|=0.
\end{align}
By the same token for the proof of Lemma~\ref{lemma-3-1}, we show that the
system
\begin{align}
\dot{v}(t)=\frac{1}{\beta^*} \partial_1 f(\sigma^*)  v(t)+\frac{1}{\beta^*} \partial_2 f(\sigma^*) v(t-\beta^* z_{\sigma^*})\label{new-eq-3-32}
\end{align}
with $\partial_i f(\sigma^*),\, \partial_i g(\sigma^*),  i=1,\,2,$ defined at (\ref{eq-3-1}), has a nonconstant periodic solution which contradicts the
assumption that $B(u_0,\,\sigma_0,\,\beta_0; r,\,\rho)$ is a special neighborhood which contains an
isolated center of (\ref{new-eq-3-5}).  Since (\ref{new-eq-3-29}) implies
$   \lim_{k\rightarrow +\infty}\|y_k-y_{\sigma_k}\|=0$,  by the same argument in the proof of Lemma~\ref{new-homotopy-lemma} we know that system~\ref{new-eq-3-32} have a nonconstant periodic solution. This is a contradiction. \hfill{ }\qed
\end{proof}
Now we  are in the position to  prove the local Hopf bifurcation
theorem for system~(\ref{SDDE-general-original}).
\begin{theorem}\label{new-localhopf}
Assume $($S1$\,)$--$($S3$\,)$ hold. Let $(x_{\sigma_0},\, \sigma_0)$ be
an isolated center of system
$($\ref{eq-3-1}$\,)$.   If  the crossing number defined by $($\ref{crossingnumber}$\,)$ satisfies
\[\gamma(x_{\sigma_0},\, \sigma_0, \,\beta_0)\neq
0,\] then there exists a bifurcation of nonconstant periodic
solutions of $($\ref{SDDE-general-original}$\,)$ near
$(x_{\sigma_0}, \,\sigma_0)$. More precisely, there exists a
sequence  $\{(x_n, \,\sigma_n,
\beta_n)\}$ such that $\sigma_n\rightarrow\sigma_0$,
$\beta_n\rightarrow\beta_0$ as $n\rightarrow\infty$,  and
$\lim_{n\rightarrow\infty}\|x_n-x_{\sigma_0}\|_{C^1}=0$, where
\[
(x_n,\, \sigma_n)\in C^1(\mathbb{R};\mathbb{R}^{N})\times\mathbb{R}
\] is a nonconstant $2\pi/\beta_n$-periodic solution of system  $($\ref{SDDE-general}$\,)$.
\end{theorem}
\begin{proof}  
%We first normalize the period of the periodic solutions of
%system (\ref{SDDE-general}) to $2\pi$ 
Let $(x,\,\tau)$ be a solution of system $($\ref{SDDE-general}$\,)$ with $x$ being $2\pi/\beta$-periodic and $\beta>0$. Let $(x(t),\tau(t))=(y(\beta t),\,z(\beta t))$. Then
system $($\ref{SDDE-general-original}$\,)$ is transformed to
\begin{align}\label{new-eq-3-22-b}
\left\{
\begin{array}{ll}
\dot{y}(t)&=\frac{1}{\beta}f(y(t),\,y(t-\beta z(t)),\,\sigma),\\
z(t)&= g(y(t),\,y(t-\beta z(t)),\,\sigma).\\
\end{array}
\right.
\end{align}
Then $x$ is a ${2\pi}/{\beta}$-periodic solution of system $($\ref{SDDE-general}$\,)$  
if and only if $y$ is a $2\pi$-periodic solution of system  (\ref{new-eq-3-22-b}). 
Let  $W=O(\epsilon_0)\cap C_{2\pi}^1(\mathbb{R};\,\mathbb{R}^{N})$. For any $\xi=e^{i\nu}\in S^1$, $u\in W$, $(\xi u)(t):=u(t+\nu)$. The idea of the proof in the sequel
is to verify all the conditions (A1)-(A6) listed in the proof of Theorem~\ref{localhopf} for applying  Theorem~2.4 on Hopf bifurcation developed in \cite{MR2644135}.
 
Recall that $\delta$ and $\varepsilon$ are defined before (\ref{crossingnumber}). Let
$\mathscr{D}(\sigma_0,\,\beta_0)=(\sigma_0-\delta,\,\sigma_0+\delta)\times
(\beta_0-\varepsilon,\,\beta_0+\varepsilon)$ and define the maps
\begin{align*} 
\begin{aligned}
 L_0y(t): & = \dot{y}(t),\, y\in C_{2\pi}^1(\mathbb{R};\,\mathbb{R}^{N}),\\
 N_1(y, \,\sigma,\,\beta)(t):& = f(y(t),\,y(t-\beta z(t)),\sigma), y\in W,\\
 \tilde{N}_1(y, \,\sigma,\,\beta)(t): & = \partial_1 f(\sigma)(y(t)-y_{\sigma})+\partial_2 f(\sigma)(y(t-\beta z_{\sigma})-y_{\sigma}), y\in W,
\end{aligned}
\end{align*}
where 
$(\sigma,\,\beta)\in\mathscr{D}(\sigma_0,\,\beta_0)$ and
$t\in\mathbb{R}$, and $(y_{\sigma},\,z_{\sigma})$ is
the stationary state of the system at $\sigma$ such that $y_{\sigma_0}=x_{\sigma_0}$. The space $C_{2\pi}^1(\mathbb{R};\,\mathbb{R}^{N})$ is a Banach
representation of the group $G=S^1$.
 
Define the operator $K: C_{2\pi}^1(\mathbb{R};\,\mathbb{R}^{N})\rightarrow
\mathbb{R}^{N}$ by
\begin{align*} 
K(y):=\frac{1}{2\pi}\int_0^{2\pi}y(t) dt,\,y\in C_{2\pi}^1(\mathbb{R};\,\mathbb{R}^{N}).
\end{align*}
By Lemma~\ref{bounded-inverse}, the operator $L_0+K: C_{2\pi}^1(\mathbb{R};\mathbb{R}^{N})\rightarrow C_{2\pi}(\mathbb{R};\,\mathbb{R}^{N})$ has a compact inverse
$(L_0+K)^{-1}$.
Then, finding a $2\pi/\beta$-periodic solution for the system
(\ref{SDDE-general-original}) is equivalent to finding a solution of the following
fixed point problem:
\begin{align}
y=(L_0+K)^{-1}\left[\frac{1}{\beta}N_1(y,\,\sigma,\,\beta)+K(y)\right],
\end{align}
where $(y,\,\sigma,\,\beta)\in W\times\mathbb{R}\times(0,\,+\infty)$.

By (S1) we know that the linear operator $\tilde{N}_1$ is
continuous. By Lemma~\ref{new-new-bounded-N-0}, we know that
$N_1(\cdot,\,\sigma,\,\beta): W\rightarrow C_{2\pi}^1(\mathbb{R};\mathbb{R}^{N})$ is continuous.
Moreover, by Lemma~\ref{new-bounded-inverse} the operator $(L_0+K)^{-1}:
C_{2\pi}(\mathbb{R};\mathbb{R}^{N}) \rightarrow C_{2\pi}^1(\mathbb{R};\mathbb{R}^{N})$ is continuous. Noticing that the embedding $j: C_{2\pi}^1(\mathbb{R};\mathbb{R}^{N})\hookrightarrow C_{2\pi}(\mathbb{R};\mathbb{R}^{N})$ is compact, we obtain that
$(L_0+K)^{-1}\circ(\frac{1}{\beta}N_1(\cdot, \alpha, \beta )+K):
W\rightarrow C_{2\pi}^1(\mathbb{R};\mathbb{R}^{N})$ and
$(L_0+K)^{-1}\circ(\frac{1}{\beta}\tilde{N}_1(\cdot, \alpha,
\beta)+K): W\rightarrow C_{2\pi}^1(\mathbb{R};\mathbb{R}^{N})$ are completely continuous and hence are
condensing maps. That is, (A2) and (A4) are satisfied.

Define the following maps $\mathcal{F}:
W\times\mathbb{R}\times(0,\,+\infty)\rightarrow C_{2\pi}^1(\mathbb{R};\mathbb{R}^{N})$ and $\tilde{\mathcal{F}}:
W\times\mathbb{R}\times(0,\,+\infty)\rightarrow C_{2\pi}^1(\mathbb{R};\mathbb{R}^{N})$ by
\begin{align*}
&\mathcal{F}(y,\,\sigma,\,\beta):=y-(L_0+K)^{-1}\left[\frac{1}{\beta}N_1(y,\,\sigma,\,\beta)+K(y)\right],\\
&\widetilde{\mathcal{F}}(y,\,\sigma,\,\beta):=y-(L_0+K)^{-1} \left[\frac{1}{\beta}
\tilde{N}_1(y, \,\sigma,\,\beta)+K (y)\right],%\label{eq-3-46}
\end{align*}
which are equivariant condensing fields. Finding a
${2\pi}/{\beta}$-periodic solution of  system (\ref{SDDE-general}) is
equivalent to finding the solution of the problem
\begin{align*}%\notag%\label{operatorequation}
\mathcal{F}(y,\,\sigma,\,\beta)=0,\quad (y, \,\sigma,\,\beta)\in
W\times\mathbb{R}\times(0,\,+\infty).
\end{align*}
Since $(x_{\sigma_0}, \,\sigma_0)=(y_{\sigma_0},\,\sigma_0)$ is an isolated center of system (\ref{eq-3-1}) with a purely imaginary characteristic value $i\beta_0$, $\beta_0>0$, $(y_{\sigma_0},\,\sigma_0,\,\beta_0)\in W\times\mathbb{R}\times(0,\,+\infty)$ is an isolated singular point of
$\tilde{\mathcal{F}}$. That is,   $(y_{\sigma_0},\,\sigma_0,\,\beta_0)$ is the only point in $W$ such that the derivative $D_y\mathcal{F}(y_{\sigma_0},\,\sigma_0,\,\beta_0)$ is not an automorphism of $C_{2\pi}^1(\mathbb{R};\mathbb{R}^{N})$.  One can define the following two-dimensional
submanifold $M\subset V^G\times\mathbb{R}\times(0,\,+\infty)$ by
\begin{align*}
M:=\{(y_\sigma,\,\sigma,\,\beta): \sigma\in
(\sigma_0-\delta,\,\sigma_0+\delta),\,\beta\in
(\beta_0-\varepsilon,\,\beta_0+\varepsilon)\},
\end{align*}
such that the point $(y_{\sigma_0},\,\sigma_0,\,\beta_0)$ is the
only singular point  of $\tilde{\mathcal{F}}$ in $M$. $M$ is the set of
trivial solutions to the system (\ref{eq-3-1}) and satisfies the
assumption (A3).

Since $(y_{\sigma_0},\,\sigma_0,\,\beta_0)\in W\times\mathbb{R}\times(0,\,+\infty)$ is an isolated singular point of
$\tilde{\mathcal{F}}$, for $\rho>0$ sufficiently small, the linear operator
$D_u \tilde{\mathcal{F}}(y_\sigma,\,\sigma,\,\beta): W\rightarrow C_{2\pi}^1(\mathbb{R};\mathbb{R}^{N})$ with $|(\sigma,\,\beta)-(\sigma_0,\,\beta_0)|<\rho$, is not an automorphism only if $(\sigma,\,\beta)=(\sigma_0,\,\beta_0)$.
Then, by the Implicit Function Theorem, there exists $r>0$ such that for every $(y,\,\sigma,\,\beta)\in W\times\mathbb{R}\times(0,\,+\infty)$ with $|(\sigma,\,\beta)-(\sigma_0,\,\beta_0)|=\rho$ and $0<\|y-y_\sigma\|\leq r$, we have
$\tilde{\mathcal{F}}(y,\,\sigma,\,\beta)\neq 0$.  Then the set $B(x_0,\,\sigma_0,\,\beta_0; r, \rho)$ defined by
\begin{align*}
\{(y,\, \sigma,\,\beta)\in W\times\mathbb{R}\times(0,\,+\infty);
|(\sigma,\,\beta)-(\sigma_0,\,\beta_0)|<\rho, \|y-y_\sigma\|_{C^1}<r\},
\end{align*} is a special neighborhood for $\tilde{\mathcal{F}}$.

By Lemma~\ref{new-lemma-3-1}, there exists a special neighborhood
$\mathcal{U}=B(y_{\sigma_0},\,\sigma_0,\,\beta_0;r', \rho)$ such that
$\mathcal{F}$  and $\tilde{\mathcal{F}}$ are
nonzero for
$(y,\,\sigma,\,\beta)\in\overline{B(y_{\sigma_0},\,\sigma_0,\,\beta_0;r',
\rho)}$ with $y\neq y_\sigma$ and
$|(\sigma,\,\beta)-(\sigma_0,\,\beta_0)|=\rho$. That is, (A5) is satisfied.

Let $\theta$ be a completing function on $\mathcal{U}$. It
follows from Lemma~\ref{new-homotopy-lemma} that $
(\mathcal{F},\,\theta) $ is homotopic to
$(\tilde{\mathcal{F}},\,\theta)$ on
$\mathcal{U}$. It is known that $C_{2\pi}^1(\mathbb{R};\mathbb{R}^{N})$ has the following isotypical direct sum
decomposition
\begin{align*}
%\displaystyle
C_{2\pi}^1(\mathbb{R};\mathbb{R}^{N})=\overline{\bigoplus\limits_{k=0}^{\infty} V_k},
\end{align*}
where $V_0$ is the space of all constant mappings from $\mathbb{R}$ into $\mathbb{R}^{N}$, and   $V_k$ with $k>0$, $k\in\mathbb{N}$ is the vector space of all mappings of the form 
\[
x\cos k\cdot+y\sin k\cdot: \mathbb{R}\ni t\rightarrow x\cos k t+y\sin k t\in\mathbb{R}^{N},
\] where  $x,\,y\in \mathbb{R}^{N}$.
Then  $V_k$, $k>0,\,k\in\mathbb{N}$, are finite dimensional.  Then,  (A1) is satisfied.

The verification of (A6) and  the computation of the crossing number $\gamma(y_{\sigma_0},\,\sigma_0,\,\beta_0)\neq
0$ is the same as that in the proof of Theorem~\ref{localhopf}. We omit the details here. %
Then by Theorem~2.4 of  \cite{MR2644135}, $
(y_{\sigma_0},\,\sigma_0,\,\beta_0)$ is a bifurcation point of the
system (\ref{new-eq-3-22-b}). Consequently, there exists a sequence of
non-constant periodic solutions
$(x_n,\,\sigma_n,\,\beta_n)$
such that $\sigma_n\rightarrow \sigma_0$, $\beta_n\rightarrow
\beta_0$ as $n\rightarrow\infty$, and   $x_n$ is a
$2\pi/\beta_n$-periodic solution of (\ref{SDDE-general-original}) such that $x_n$ satisfies (\ref{SDDE-general-original}) with
$\lim_{n\rightarrow+\infty}\|x_n-x_{\sigma_0}\|_{C^1}=0$.
\hfill{ }\qed \end{proof}

\section{Global Bifurcation of DAEs with State-dependent Delays}\label{global-bifurcation}

In this section we use Rabinowitz type global Hopf bifurcation Theorem~2.5 developed in \cite{MR2644135} to describe the maximal
continuation of bifurcated periodic solutions with large amplitudes
when the bifurcation parameter $\sigma$ is far away from the
bifurcation value. Note that systems~$($\ref{SDDE-general-original}$)$ and $($\ref{SDDE-general}$)$ share the same differential equation for $x$ and differ only in the algebraic equation for the state-dependent delay $\tau$. Moreover, by Theorems~\ref{localhopf}
 and \ref{new-localhopf}, both systems share the same set of Hopf bifurcation points. In the following, we state results in terms of system~$($\ref{SDDE-general-original}$)$, which are also applicable to system~$($\ref{SDDE-general}$)$.

\begin{lemma}[Vidossich, \cite{Vidossich}]\label{vido}
Let $X$ be a Banach space, $v: \mathbb{R}\rightarrow X$ be a
$\mathsf{p}$-periodic function with the following properties:
\begin{enumerate}
\item[(i)] $v\in L_{loc}^1(\mathbb{R},\, X)$;
\item[(ii)] there exists $U\in L^1([0,\,\frac{\mathsf{p}}{2}];\mathbb{R}_+)$
such that $|v(t)-v(s)|\leq U(t-s)$ for almost every $($in the sense of the Lebesgue measure$)$ $s,\,t\in\mathbb{R}$ such that $s\leq t$,
$t-s\leq \frac{\mathsf{p}}{2}$;
\item[(iii)] $\int_0^{\mathsf{p}} v(t)\,\mathrm{d}t=0$.
\end{enumerate}
Then \[ \mathsf{p}\,\|v\|_{L^{\infty}}\leq
2\int_0^{\frac{\mathsf{p}}{2}}U(t)\,\mathrm{d}t.\]
\end{lemma}

We make the following assumption on $f$:
 \begin{enumerate}
\item[(S4)] There exists constant $L_f>0$ such that
\begin{align*}
&
|f(\theta_1,\,\theta_2,\,\sigma)-f(\overline{\theta}_1,\,\overline{\theta}_2,\,\sigma)|\leq
L_f
(|\theta_1-\overline{\theta}_1|+|\theta_2-\overline{\theta}_2|), 
\end{align*}
for every
$\theta_1,\,\theta_2,\,\overline{\theta}_1,\,\overline{\theta}_2,\, \sigma\in\mathbb{R}$.
\end{enumerate}

\begin{lemma}\label{period-bound}
Suppose that system $($\ref{SDDE-general-original}$\,)$ satisfies  the assumption
$($S4$\,)$ and $x$ is a nonconstant periodic
solution. The following statements are true.
\begin{enumerate}
\item[i)] If $\|\tau\|_{L^{\infty}}<\frac{1}{2L_f}$, then  the minimal period $\mathsf{p}$ of $x$ satisfies
\[
\mathsf{p}\geq
\frac{2}{1-2L_f\|\tau\|_{L^{\infty}}} .
\]
\item[ii)]  If $\tau$ is continuously differentiable in $\mathbb{R}$, then the minimal period $\mathsf{p}$ of $x$ satisfies
\[
\mathsf{p}\geq
\frac{4}{L_f(2+|\dot{\tau}|_{L^{\infty}})}.
\]

\item[iii)] Suppose there exists a  constant $L_g>0$ such that
\begin{align*}
&
|g(\theta_1,\,\theta_2,\,\sigma)-g(\overline{\theta}_1,\,\overline{\theta}_2,\,\sigma)|\leq
L_g
(|\theta_1-\overline{\theta}_1|+|\theta_2-\overline{\theta}_2|), 
\end{align*}
for every
$\theta_1,\,\theta_2,\,\overline{\theta}_1,\,\overline{\theta}_2,\, \sigma\in\mathbb{R}$.
If $\|\dot{x}\|_{L^{\infty}}<\frac{1}{L_g}$, then  the minimal period $\mathsf{p}$ of $x$ satisfies
\[
\mathsf{p}\geq
 \frac{2(1-L_g|\dot{x}|_{L^{\infty}})}{L_f}.
\]

\end{enumerate}
\end{lemma}
\begin{proof}  Assume that  $x$ is a nonconstant periodic
solution with  minimal period $\mathsf{p}$. 
Let $v(t)=\dot{x}(t)$. Then we have $\int_0^\mathsf{p} v(t)\mathrm{d}t=0$.  For $s\leq t$, by (S4)
and the Integral Mean Value Theorem, we have
\begin{align}\label{vts-ineq}
|v(t)-v(s)|&\leq |\dot{x}(t)-\dot{x}(s)|\notag\\
               &\leq L_f(|x(t)-x(s)|+|x(t-\tau(t))-x(s-\tau(s))|)\notag\\ 
               &\leq L_f|\dot{x}|_{L^{\infty}}(t-s)+  L_f|\dot{x}|_{L^{\infty}}(t-s+|\tau(t)-\tau(s)|)\notag\\ 
               &\leq \left(2L_f|\dot{x}|_{L^{\infty}} +L_f
               |\dot{x}|_{L^{\infty}}\cdot|\dot{\tau}|_{L^{\infty}}\right)(t-s).
\end{align}

i) If $\|\tau\|_{L^{\infty}}<\frac{1}{2L_f}$, then by (\ref{vts-ineq}) we have
\begin{align*}
|v(t)-v(s)|
               &\leq L_f|\dot{x}|_{L^{\infty}}(t-s)+  L_f|\dot{x}|_{L^{\infty}}(t-s+|\tau(t)-\tau(s)|)\\ 
               &\leq 2L_f |\dot{x}|_{L^{\infty}} (t-s)+2L_f |{\tau}|_{L^{\infty}} \cdot|\dot{x}|_{L^{\infty}}.
                \end{align*}
Let
\[
U(t)=
 2L_f |\dot{x}|_{L^{\infty}} t+2L_f |{\tau}|_{L^{\infty}} \cdot|\dot{x}|_{L^{\infty}}.
\]
Then, by Lemma~\ref{vido},
we obtain
\begin{align*}
\mathsf{p}|\dot{x}|_{L^{\infty}}& \leq 2 \int _0^{\frac{\mathsf{p}}{2}}U(t)\mathrm{d}t=\frac{\mathsf{p}^2}{4} \cdot 2L_f |\dot{x}|_{L^{\infty}} +\mathsf{p}\cdot 2L_f |{\tau}|_{L^{\infty}} \cdot|\dot{x}|_{L^{\infty}}.
\end{align*}Therefore,
\[
\mathsf{p}\geq
\frac{2}{ (1-2L_f |{\tau}|_{L^{\infty}})}.
\]

ii) If $\tau$ is continuously differentiable in $\mathbb{R}$, then we have $|\dot{\tau}|_{L^{\infty}}<\infty$. Moreover,  by (\ref{vts-ineq}) we have
\begin{align*}
|v(t)-v(s)|
               &\leq L_f|\dot{x}|_{L^{\infty}}(t-s)+  L_f|\dot{x}|_{L^{\infty}}(t-s+|\tau(t)-\tau(s)|)\\ 
               &\leq (2 +|\dot{\tau}|_{L^{\infty}}) L_f\cdot |\dot{x}|_{L^{\infty}}(t-s).
                \end{align*}
Let
\[
U(t)=
 (2 +|\dot{\tau}|_{L^{\infty}})L_f \cdot |\dot{x}|_{L^{\infty}}t.
\]
Then, by Lemma~\ref{vido},
we obtain
\begin{align*}
\mathsf{p}|\dot{x}|_{L^{\infty}}& \leq 2 \int _0^{\frac{\mathsf{p}}{2}}U(t)\mathrm{d}t=\frac{\mathsf{p}^2}{4} \cdot  (2 +|\dot{\tau}|_{L^{\infty}}) L_f\cdot |\dot{x}|_{L^{\infty}}.
\end{align*}Therefore,
\[
\mathsf{p}\geq
\frac{4}{ L_f(2 +|\dot{\tau}|_{L^{\infty}})}.
\]

iii)  If $g$ is Lipschitz continuous, then  we have
\begin{align}
|\tau(t)-\tau(s)| & \leq L_g |x(t)-x(s)|+L_g |x(t-\tau(t))-x(s-\tau(s))|\notag\\
&\leq   L_g|\dot{x}|_{L^{\infty}}(t-s)+L_g |\dot{x}|_{L^{\infty}}(t-s+|\tau(t)-\tau(s)|)\notag.
\end{align}If   $|\dot{x}|_{L^{\infty}}<\frac{1}{L_g}$, then we have
\begin{align}\label{tauts-ineq}
|\tau(t)-\tau(s)| \leq \frac{2L_g|\dot{x}|_{L^{\infty}}(t-s)}{1-L_g|\dot{x}|_{L^{\infty}}}.
\end{align}By (\ref{vts-ineq}) and (\ref{tauts-ineq})
\begin{align*}
|v(t)-v(s)|
               &\leq L_f|\dot{x}|_{L^{\infty}}(t-s)+  L_f|\dot{x}|_{L^{\infty}}(t-s+|\tau(t)-\tau(s)|)\\ 
               &\leq 2L_f\cdot |\dot{x}|_{L^{\infty}}(t-s)+  \frac{2L_fL_g|\dot{x}|_{L^{\infty}}^2(t-s)}{1-L_g|\dot{x}|_{L^{\infty}}}\\
               & =\frac{2L_f |\dot{x}|_{L^{\infty}}}{1-L_g|\dot{x}|_{L^{\infty}}}(t-s).
                \end{align*}
Let
\[
U(t)=
 \frac{2L_f |\dot{x}|_{L^{\infty}}}{1-L_g|\dot{x}|_{L^{\infty}}}t.
\]
We obtain 
\begin{align*}
\mathsf{p}|\dot{x}|_{L^{\infty}}& \leq 2 \int _0^{\frac{\mathsf{p}}{2}}U(t)\mathrm{d}t=\frac{\mathsf{p}^2}{4} \cdot \frac{2L_f |\dot{x}|_{L^{\infty}}}{1-L_g|\dot{x}|_{L^{\infty}}},
\end{align*}
and
\begin{align*}
\mathsf{p}\geq \frac{2(1-L_g|\dot{x}|_{L^{\infty}})}{L_f}.
\end{align*}\hfill{ }\qed
\end{proof}
To describe the minimal periods of the periodic solutions near the bifurcation point, we need the following result which was first established in \cite{MY} for
ordinary differential equations and was extended to other types of delay differential equations in \cite{Wu-1, MR2644135}.
\begin{lemma}\label{virtual-period}
Suppose that system $($\ref{SDDE-general-original}$\,)$ satisfies  $($S1--S4$\,)$. Assume further that there exists a sequence of real numbers
$\{\sigma_k\}_{k=1}^{\infty}$ such that:
\begin{enumerate}
\item[(i)] For each $k$, system $($\ref{SDDE-general}$\,)$ with
$\sigma=\sigma_k$ has a nonconstant periodic solution
$x_k\in C(\mathbb{R};\mathbb{R}^{N+1})$ with the minimal period $T_k>0,$ and one of the conditions   i), ii) and iii) at Lemma~\ref{period-bound} is satisfied by $(x_k,\,\tau_k)$;
\item[(ii)] $\lim\limits_{k\rightarrow\infty}\sigma_k=\sigma_0\in\mathbb{R}$,
$\lim\limits_{k\rightarrow\infty}T_k=T_0<\infty $, and
$\lim\limits_{k\rightarrow\infty}\|x_k-x_0\|=0$,  where
$x_0:\mathbb{R}\rightarrow\mathbb{R}^{N}$ is a constant map with the value  $x_{0}$.

\end{enumerate}
Then $x_0$ is a stationary state of $($\ref{SDDE-general-original}$\,)$ and
there exists $m\geq 1,\,m\in\mathbb{N}$ such that $\pm im\,{2\pi }/{T_0}$ are the roots of the
characteristic equation $($\ref{eq-3-2}$\,)$ with $\sigma=\sigma_0$.
\end{lemma}
\begin{proof}
By Lemma~\ref{period-bound} and the uniform convergence of $\{(x_k,\,\sigma_k,\,T_k)\}_{k=1}^\infty$ we conclude that there exists $T^*>0$ such that $T_k\geq T^*$ and therefore $T_0\geq T^*$. We can show that $(x_0,\,\sigma_0)$ is a stationary state of
$($\ref{SDDE-general}$\,)$,  and that  the following linear system
\begin{align}
\dot{v}(t)= 
\partial_1 f(\sigma_0)v(t)+ \partial_2 f(\sigma_0)v(t-\tau_0)\label{eq-4-1}
\end{align}
has a nonconstant periodic solution, the proofs of which are just simplified versions of the proof for Lemma~4.3 in \cite{ MR2644135} without the equations for $\tau_k$. Hence we omit the details here.  Then by  Lemma~\ref{linear-eq-center}, 
 there exists $m\geq 1,\, m\in\mathbb{N}$, such that $\pm im\,2\pi/T_0$ are characteristic
values of (\ref{eq-3-2}). This completes the proof.
\hfill{ }\qed
\end{proof}

Now we can describe the relation between $2\pi/\beta_k$ and the
minimal period of $u_k$ in Theorem~\ref{localhopf}.
\begin{theorem}Assume $($S1--S4$\,)$ hold and    every point in the sequence $\{(x_k,\,\tau_k)\}_{k=1}^\infty$ at Theorem~\ref{localhopf} satisfies one of the conditions among  i), ii) and iii) at Lemma~\ref{period-bound}, then every limit point of the minimal period of
$x_k$ as $k\rightarrow+\infty$ is contained in the set
\begin{align*}
\left\{\frac{2\pi}{(n\beta_0)}:\pm im\,n\beta_0 \mbox{ are
characteristic values of $(x_0,\,\sigma_0)$,}\,m,\,n\geq 1,\,m,\,n\in\mathbb{N}\right\}.
\end{align*}
Moreover, if $\pm i\,m\,n\beta_0$ are not characteristic values
of $(x_0,\,\sigma_0)$ for any integers $m,\, n\in\mathbb{N}$ such that $m\,n>1$,  then $2\pi/\beta_k$ is
the minimal period of $u_k(t)$ and
$2\pi/\beta_k\rightarrow2\pi/\beta_0$ as $k\rightarrow\infty$.
\end{theorem}
\begin{proof}
Let $T_k$ denote the minimal period of $x_k(t)$. Then there exists a
positive integer $n_k$ such that $2\pi/\beta_k=n_k T_k$. Since
$T_k\leq2\pi/\beta_k\rightarrow2\pi/\beta_0$ as
$k\rightarrow\infty$, there exists a subsequence
$\{T_{k_j}\}_{j=1}^{\infty}$ and $T_0$ such that
$T_0=\lim_{j\rightarrow\infty}T_{k_j}$. Since $2\pi/\beta_{k_j}\rightarrow2\pi/\beta_0$, $T_{k_j}\rightarrow
T_0$ as $j\rightarrow\infty,$ $n_{k_j}$ is identical to a constant
$n$ for $k$ large enough. Therefore, $2\pi/\beta_0=n T_0$. Thus
$T_{k_j}\rightarrow 2\pi/(n\beta_0)$ as $j\rightarrow\infty$. By
Lemma~\ref{virtual-period}, $\pm im\,{2\pi}/{T_0}=\pm im\,n\beta_0$ are 
characteristic values of $(x_0,\,\sigma_0)$ for some $m\geq 1,\,m\in\mathbb{N}$.

Moreover, if $\pm i\,m\,n\beta_0$ are not characteristic values
of $(u_0,\,\sigma_0)$ for any integers $m\in\mathbb{N}$ and $n\in\mathbb{N}$ with $mn>1$, then $m=n=1$. Therefore, for $k$ large enough $n_{k_j}=1$ and $2\pi/\beta_k=T_k$ is the minimal period of $x_k(t)$ and
$2\pi/\beta_k\rightarrow2\pi/\beta_0$ as $k\rightarrow\infty$.
This completes the proof.\hfill { }\qed
\end{proof}
The following lemma shows that we can locate all the possible Hopf
bifurcation points of   system (\ref{SDDE-general-original}) with state-dependent
delay at the centers of its corresponding formal linearization. Since the proof is similar to that  for Lemma~4.5 in \cite{MR2644135}, we omit the details here.
\begin{lemma}\label{lemma-4-5}
Assume (S1--S3) hold. If $(x_0,\,\sigma_0)$ is a Hopf bifurcation point
of   system (\ref{SDDE-general-original}), then it is a center of $($\ref{eq-3-1}$\,)$.
\end{lemma} 

Now we are in the position to consider the global Hopf bifurcation problem of
system (\ref{SDDE-general-original}). Letting
$(x(t),\tau(t))=(y(\frac{2\pi}{\mathsf{p}}t),\,z(\frac{2\pi}{\mathsf{p}}t))$,
we can reformulate the problem as a problem of finding $2\pi$-period
solutions to the following equation:
\begin{align}\label{p-sys}
\dot{y}(t)=\frac{\mathsf{p}}{2\pi}N_0(y(t),\,\sigma,\,{2\pi}/{\mathsf{p}}),
\end{align}where the $z$ satisfies the algebraic equation $z(t)=g(y(t),y(t-\frac{\mathsf{p}}{2\pi}z(t)),\,\sigma)$. 
 Accordingly, the formal linearization
(\ref{eq-3-1}) becomes
\begin{align}\label{p-formal-linearization}
\dot{x}(t)=\frac{\mathsf{p}}{2\pi}\tilde{N}_0(x(t),\,\sigma,\,{2\pi}/{\mathsf{p}}).
\end{align}
Using the same notations as in the proof of Theorem~\ref{localhopf},
we can define \[ \mathscr{N}_0(x,\,\sigma, \,
\mathsf{p})=N_0(x,\,\sigma, \,{2\pi}/{\mathsf{p}}), \tilde{\mathscr{N}}_0(x,\,\sigma,
\, \mathsf{p})=\tilde{N}_0(x,\,\sigma, \,{2\pi}/{\mathsf{p}}).\]

Then the following
system
\begin{align}\label{abs-global-hopf}
L_0 x=\frac{\mathsf{p}}{2\pi}\mathscr{N}_0(x,\,\sigma,\, {\mathsf{p}}),\,\mathsf{p}>0,
\end{align}
is equivalent to (\ref{p-sys}) and
\begin{align}\label{abs-global-hopf-1}
L_0 x=\frac{\mathsf{p}}{2\pi}\tilde{\mathscr{N}}_0(x,\,\sigma,\,
{\mathsf{p}}),\,\mathsf{p}>0,
\end{align}
is equivalent to (\ref{p-formal-linearization}).
Let $\mathcal{S}$ denote the closure of the set of all nontrivial
periodic solutions of system~(\ref{abs-global-hopf}) in the space
$V\times\mathbb{R}\times\mathbb{R}_+$, where  $\mathbb{R}_+$ is the set of all nonnegative reals.
It follows from Lemma~\ref{period-bound} that the constant solution $(x_0,\,\sigma_0,\,0)$ does not belong
to this set if the sequence $\{(x_k,\,\tau_k)\}_{k=1}^\infty$ in Theorem~\ref{localhopf} satisfies one of the conditions among  i), ii) and iii) at Lemma~\ref{period-bound}. Consequently, we can assume that problem
(\ref{abs-global-hopf}) is well posed on the whole space
$V\times\mathbb{R}^2$, in the sense that if  $\mathcal{S}$ exists in $V\times\mathbb{R}^2$, then it must be contained in $V\times\mathbb{R}\times\mathbb{R}_+$.

%On the other hand, assume (S3) holds at every center of
%$($\ref{abs-global-hopf-1}$\,)$. Then,  from the proof of
%Theorem~\ref{localhopf} we know that the assumptions (S1-S3) are
%sufficient for the systems $($\ref{abs-global-hopf}$\,)$  and
%$($\ref{abs-global-hopf-1}$\,)$ to satisfy the conditions (A1--A6).
%Also, under the same assumptions, Lemma~\ref{lemma-4-5} implies (A7)
% and Lemma~\ref{homotopy-lemma} implies (A8).  
 Then by the   global Hopf bifurcation theorem 2.5 developed in  \cite{MR2644135} and with similar arguments leading to Theorem~4.6 in  \cite{MR2644135}, we
obtain the following global Hopf bifurcation theorem for system
$($\ref{SDDE-general-original}$\,)$ with state-dependent delay.
\begin{theorem}\label{global-new-th}
Suppose that  system $($\ref{SDDE-general-original}$\,)$ satisfies $($S1-S4$\,)$ and (S3) holds at every center of $($\ref{abs-global-hopf-1}$\,)$.
Assume that all the centers of
$($\ref{abs-global-hopf-1}$\,)$ are isolated and every periodic solution $x$ of  system $($\ref{SDDE-general-original}$\,)$ satisfies one of the conditions among  i), ii) and iii) at Lemma~\ref{period-bound}. Let $M$ be the set of trivial periodic solutions of
$($\ref{abs-global-hopf}$\,)$ and $M$ is complete.  If
$(x_0,\,\sigma_0,\,\mathsf{p}_0)\in M$ is a bifurcation point, then
either the connected component ${C}(x_0,\,\sigma_0,\,\mathsf{p}_0)$
of $(x_0,\,\sigma_0,\,\mathsf{p}_0)$ in $\mathcal{S}$ is unbounded,
or
\[
{C}(x_0,\,\sigma_0,\,\mathsf{p}_0) \cap
M=\{(x_0,\,\sigma_0,\,\mathsf{p}_0),\,(x_1,\,\sigma_1,\,\mathsf{p}_1),\cdots,
(x_q,\,\sigma_q,\,\mathsf{p}_q)\},
\]
where $\mathsf{p}_i\in\mathbb{R}_+$,
$(x_i,\,\sigma_i,\,\mathsf{p}_i)\in M$, $i=0,\,1,\,2,\cdots, q$.
Moreover, in the latter case, we have
\[
\sum\limits_{i=0}^{q}\epsilon_i\gamma(x_i,\,\sigma_i,\,2\pi/\mathsf{p}_i)=0,
\]
where $\gamma(x_i,\,\sigma_i,\,2\pi/\mathsf{p}_i)$ is the crossing number of
$(x_i,\,\sigma_i,\,\mathsf{p}_i)$ defined by $($\ref{crossingnumber}$\,)$ and 
 \[
 \epsilon_i=\mbox{$\mathrm{sgn}$}\det  (\partial_1 f(\sigma_i)+\partial_2 f(\sigma_i)).
 \]
\end{theorem}

\section{Global Hopf bifurcation of a model of regulatory dynamics}\label{Regulatory-Model}
 We consider the following extended Goodwin's model for regulatory dynamics:
\begin{align}\label{SDDE-hes1} 
\left\{
\begin{aligned}
\frac{\mathrm{d}x(t)}{\mathrm{d}t} & =-\mu_m x(t)+\frac{\alpha_m}{1+\left(\frac{z(t-\tau)}{\tilde{z}}\right)^h},\\
  \frac{\mathrm{d}y(t)}{\mathrm{d}t} & =-\mu_p y(t)+\alpha_p x(t-\tau),\\
   \frac{\mathrm{d}z(t)}{\mathrm{d}t} & =-\mu_e z(t)+\alpha_e y(t-\tau),\\
 \tau(t)&=c(x(t)-x(t-\tau)),
\end{aligned}
\right.
\end{align}
 where $x$ is the  concentration of mRNA, $y$ is  the  concentration of the related protein; $z$ is the concentration of an active enzyme which controls the level of the metabolite functioning as repressor at the DNA level; $\mu_m$,  $\mu_p$ and $\mu_e$ are nonnegative degradation rates; $\alpha_m$, $\alpha_p$ and $\alpha_e$  are positive coefficients for the inhibition/activation terms;  $c$ and $\tilde{z}$ are positive constants; $h$ is an even positive integer.  The  Goodwin's model \cite{Goodwin} without delay ($\tau=0$) has been extensively studied in system biology modeling various regulatory dynamics. Note that if we freeze the delay $\tau$ at the stationary state in system~(\ref{SDDE-hes1}), it becomes the classic Goodwin's model without delay.
 
We are interested in the onset and termination of each Hopf bifurcation branch of periodic solutions  which are described as one of  the alternatives given in Theorem~\ref{global-new-th}. 
To be specific, we need to obtain the boundedness or unboundedness of the connected component of  the pairs of nonconstant periodic solution and parameter in the product space of the state and the parameter space.  In the following, we first analyze the local Hopf bifurcation of system~(\ref{SDDE-hes1})  and then consider  the boundedness of periodic solutions  of system~(\ref{SDDE-hes1}) for a  global Hopf bifurcation in light of Theorem~\ref{global-new-th}. 
 
 \subsection{Local Hopf bifurcation}
 Note that $h$ is an even positive integer. Every stationary point $(x,\,y,\,z)$ of  System~(\ref{SDDE-hes1}) satisfies that
 \begin{align}\label{SDDE-hes1-equi} 
\left\{
\begin{aligned}
 -\mu_m x+\frac{\alpha_m}{1+\left(\frac{z}{\tilde{z}}\right)^h}= &\, 0,\\
  -\mu_p y+\alpha_p x= &\, 0,\\
   -\mu_e z+\alpha_e y= &\, 0,
  \end{aligned}
\right.
\end{align}  and $(x,\,y,\,z)=\left(x_0,\,\frac{\alpha_p}{\mu_p} x_0,\,\frac{\alpha_e\alpha_p}{\mu_e\mu_p} x_0\right)$, where by Descartes' rule of signs we know that $x=x_0$ is the unique solution of \[
\mu_m\left(\frac{\alpha_e\alpha_p}{\mu_e\mu_p\tilde{z}}\right)^hx^{h+1}+\mu_m x-\alpha_m=0.
\]  Freezing the delay of system~(\ref{SDDE-hes1}) at $\tau=0$ and linearizing the resulting nonlinear system  at the stationary state $(x,\,y,\,z)=\left(x_0,\,\frac{\alpha_p}{\mu_p} x_0,\,\frac{\alpha_e\alpha_p}{\mu_e\mu_p} x_0\right)$ lead to the characteristic polynomial 
\begin{align}\label{ch4-charact-eq}
& \det\left(\lambda I-\begin{bmatrix*}[c]
-\mu_m &  0 & -\frac{h\alpha_m{z}^{h-1}}{\tilde{z}^h\left(1+\left(\frac{z}{\tilde{z}}\right)^h\right)^2}\notag\\
\alpha_p  &  -\mu_p & 0\\
0 & \alpha_e & -\mu_e
\end{bmatrix*}\right)\\ =& (\lambda+\mu_m)(\lambda+\mu_p)(\lambda+\mu_e)+\frac{h\alpha_m{z}^{h-1}}{\tilde{z}^h\left(1+\left(\frac{z}{\tilde{z}}\right)^h\right)^2},
\end{align}
which has a unique negative root and a pair of imaginary roots.
In the following, we discuss the existence of purely imaginary eigenvalues as the parameter $\alpha_m$ varies.   We have 

\begin{lemma}\label{lemma-4-1}
Let $(x,\,y,\,z)$ be a stationary state of system~(\ref{SDDE-hes1}). Then
the following equation of $(x,\,\alpha_m)$
\begin{align}\label{sec4-eq-3}
\left\{
\begin{aligned}
(\mu_m+\mu_p)(\mu_e+\mu_p)(\mu_e+\mu_m)= &\frac{h\alpha_m^3}{\tilde{z}^h\mu_m^2}\cdot\left(\frac{\alpha_e\alpha_p}{\mu_e\mu_p}\right)^{h-1}{x}^{h-3},\\
\mu_mx=& \frac{\alpha_m}{1+\left(\frac{\alpha_e\alpha_p}{\mu_e\mu_p\tilde{z}}x\right)^{h}},
\end{aligned}
\right.
\end{align} has a unique solution for $(x,\,\alpha_m)=(x^*,\,\alpha_m^*)$.
\end{lemma}
\begin{proof}
% By the second equation of (\ref{sec4-eq-3}), we know that for every $\alpha_m$, there exists a unique solution for $x$ and the solution is positive. Moreover, we have
%\begin{align*}
%\frac{\mathrm{d}x}{\mathrm{d}\alpha_m}=\frac{1}{\mu_m+\mu_m(h+1)\left(\frac{\alpha_e\alpha_p}{\mu_e\mu_p\tilde{z}}\right)^hx^h}>0.
%\end{align*} Thus $\alpha_m$ is a continuously differentiable  monotone function of $x$.
 Noticing that by the second equation of (\ref{sec4-eq-3}), $\frac{\alpha_m}{x}=\mu_m\left(1+\left(\frac{\alpha_e\alpha_p}{\mu_e\mu_p\tilde{z}}x\right)^{h}\right)$, we rewrite the first equation of  (\ref{sec4-eq-3}) into
\begin{align*}
x^h\left(1+\left(\frac{\alpha_e\alpha_p}{\mu_e\mu_p\tilde{z}}x\right)^{h}\right)^3=\frac{(\mu_m+\mu_p)(\mu_e+\mu_p)(\mu_e+\mu_m)}{\frac{h\mu_m}{\tilde{z}^h}\left(\frac{\alpha_e\alpha_p}{\mu_e\mu_p\tilde{z}}\right)^{h-1}},
\end{align*}which has a unique positive solution for $x^h$ and hence for $x$ with $x=x^*$ for some $x^*>0$. Then $\alpha_m=\alpha_m^*$ with $\alpha_m^*=x^*\mu_m\left(1+\left(\frac{\alpha_e\alpha_p}{\mu_e\mu_p\tilde{z}}x^*\right)^{h}\right)$. The solution of (\ref{sec4-eq-3}) is $(x,\,\alpha_m)=(x^*,\,\alpha_m^*)$. \qed \end{proof}

\begin{lemma}Let  $\alpha_m^*$ be as in Lemma~\ref{lemma-4-1} and $\lambda=u\pm iv$ be the imaginary roots of the characteristic polynomial at (\ref{ch4-charact-eq}).  Then $u$ and $v$ are continuously differentiable with respect to $\alpha_m$ and  $u=0$ if and only if $\alpha_m=\alpha_m^*$. Moreover, 
\[
\frac{\mathrm{d}u}{\mathrm{d}\alpha_m}\,\vline_{\,\alpha_m=\alpha_m^*}>0.
\]
\end{lemma}
\begin{proof}Let $(x,\,y,\,z)=\left(x_0,\,\frac{\alpha_p}{\mu_p} x_0,\,\frac{\alpha_e\alpha_p}{\mu_e\mu_p} x_0\right)$ be a stationary state of System~(\ref{SDDE-hes1}) and let
\[
F(\lambda,\,\alpha_m)= (\lambda+\mu_m)(\lambda+\mu_p)(\lambda+\mu_e)+\frac{h\alpha_m{z}^{h-1}}{\tilde{z}^h\left(1+\left(\frac{z}{\tilde{z}}\right)^h\right)^2}.
\]Noticing that $z=\frac{\alpha_e\alpha_p}{\mu_e\mu_p} x_0$ and
\[
\frac{\mathrm{d}x_0}{\mathrm{d}\alpha_m}=\frac{1}{\mu_m+\mu_m(h+1)\left(\frac{\alpha_e\alpha_p}{\mu_e\mu_p\tilde{z}}\right)^hx_0^h},
\]we know that $F$ is continuously differentiable with respect to $(\lambda,\,\alpha_m)$. Let $(\lambda,\,\alpha_m)$ be such that $F(\lambda,\,\alpha_m)=0$. Then we have
\begin{align*}
\frac{\mathrm{d}F}{\mathrm{d}\lambda}=&(\lambda+\mu_m)(\lambda+\mu_p)(\lambda+\mu_e)\left(\frac{1}{\lambda+\mu_m}+\frac{1}{\lambda+\mu_p}+\frac{1}{\lambda+\mu_e}\right)\\
&=-\frac{h\alpha_m{z}^{h-1}}{\tilde{z}^h\left(1+\left(\frac{z}{\tilde{z}}\right)^h\right)^2}
\left(\frac{1}{\lambda+\mu_m}+\frac{1}{\lambda+\mu_p}+\frac{1}{\lambda+\mu_e}\right).
\end{align*}
Next we show that $\frac{\mathrm{d}F}{\mathrm{d}\lambda}\neq 0$ at every solution of $F(\lambda,\,\alpha_m)=0$. Otherwise, $F$ has a repeated root and the root satisfies
\[
\frac{1}{\lambda+\mu_m}+\frac{1}{\lambda+\mu_p}+\frac{1}{\lambda+\mu_e}=0
\]which lead to two distinct negative roots:
\[
\lambda=\frac{-(\mu_m+\mu_p+\mu_e)\pm\sqrt{(\mu_m+\mu_e)^2+\mu_p^2-\mu_p(\mu_m+\mu_e)}}{3}.
\]This is a contradiction. Then by the Implicit Function Theorem, $\lambda$ is continuously differentiable with respect to $\alpha_m$.

Next we bring $\lambda=u+iv$ into the  characteristic polynomial at (\ref{ch4-charact-eq}) we have
\begin{align}\label{ch4-charact-eq5}
\left\{
\begin{aligned}
&((u+\mu_m)(u+\mu_p)-v^2)(u+\mu_e)-(\mu_m+\mu_p+2u)v^2+c_0=0\\
&[(u+\mu_m)(u+\mu_p)-v^2+(u+\mu_e)(\mu_m+\mu_p+2u)]v=0,
\end{aligned}
\right.
\end{align}where $c_0=\frac{h\alpha_m{z}^{h-1}}{\tilde{z}^h\left(1+\left(\frac{z}{\tilde{z}}\right)^h\right)^2}=\frac{h\alpha_m^3}{\tilde{z}^h\mu_m^2}\cdot\left(\frac{\alpha_e\alpha_p}{\mu_e\mu_p}\right)^{h-1}{x}^{h-3}.$ If $u=0$, then (\ref{ch4-charact-eq5}) leads to
\begin{align}\label{sec4-eq-3-v-square}
\left\{
\begin{aligned}
(\mu_m+\mu_p)(\mu_e+\mu_p)(\mu_e+\mu_m)= &\, c_0,\\
 \mu_m\mu_p+\mu_e(\mu_m+\mu_p)= &\, v^2.
\end{aligned}
\right.
\end{align}where $x$ satisfies $\mu_mx= \frac{\alpha_m}{1+\left(\frac{\alpha_e\alpha_p}{\mu_e\mu_p\tilde{z}}x\right)^{h}}.$ By  Lemma~\ref{sec4-eq-3}, we have $\alpha_m=\alpha_m^*.$ By the uniqueness of $\alpha_m^*$, $u=0$ if and only if $\alpha_m=\alpha_m^*$.

To compute $\frac{\mathrm{d}u}{\mathrm{d}\alpha_m}$ at $\alpha_m=\alpha_m^*$, we take derivatives with respect to $\alpha_m$ on both sides of the equations at (\ref{ch4-charact-eq5}) and then let $u=0$, we obtain
\begin{align*}%\label{sec4-eq-3}
\left\{
\begin{aligned}
[(\mu_e\mu_p+\mu_e\mu_m+\mu_m\mu_p)-3v^2]u'-2v(\mu_m+\mu_p)v'+c_0' &=0,\\
[2(\mu_m+\mu_p+\mu_e)v]u'+[(\mu_e\mu_p+\mu_e\mu_m+\mu_m\mu_p)-3v^2]v'= &\, 0.
\end{aligned}
\right.
\end{align*}Then we have
\[
\frac{\mathrm{d}u}{\mathrm{d}\alpha_m}\,\vline_{\,\alpha_m=\alpha_m^*}=\frac{-c'_0((\mu_e\mu_p+\mu_e\mu_m+\mu_m\mu_p)-3v^2)}{[(\mu_e\mu_p+\mu_e\mu_m+\mu_m\mu_p)-3v^2]^2+4v^2(\mu_m+\mu_p)(\mu_m+\mu_p+\mu_e)}.
\]
By the second equation of (\ref{sec4-eq-3-v-square}), we have
\[
\frac{\mathrm{d}u}{\mathrm{d}\alpha_m}\,\vline_{\,\alpha_m=\alpha_m^*}=\frac{2c'_0(\mu_e\mu_p+\mu_e\mu_m+\mu_m\mu_p)}{[(\mu_e\mu_p+\mu_e\mu_m+\mu_m\mu_p)-3v^2]^2+4v^2(\mu_m+\mu_p)(\mu_m+\mu_p+\mu_e)}.
\]Noticing that
\begin{align*}
c_0 & =\frac{h\alpha_m^3}{\tilde{z}^h\mu_m^2}\cdot\left(\frac{\alpha_e\alpha_p}{\mu_e\mu_p}\right)^{h-1}{x}^{h-3}\\
 & = \frac{h}{\tilde{z}^h\mu_m^2}\cdot\left(\frac{\alpha_e\alpha_p}{\mu_e\mu_p}\right)^{h-1}{x}^{h}\left(\frac{\alpha_m}{x}\right)^3\\
 & = \frac{h\mu_m}{\tilde{z}^h}\cdot\left(\frac{\alpha_e\alpha_p}{\mu_e\mu_p}\right)^{h-1}{x}^{h}\left(1+\left(\frac{\alpha_e\alpha_p}{\mu_e\mu_p\tilde{z}}x\right)^{h}\right)^3 
\end{align*}can be regarded as a fourth order polynomial of $x^h$ with positive coefficients, and that $\frac{\mathrm{d}x_0}{\mathrm{d}\alpha_m}=\frac{1}{\mu_m+\mu_m(h+1)\left(\frac{\alpha_e\alpha_p}{\mu_e\mu_p\tilde{z}}\right)^hx_0^h}>0$, we have
\[
\frac{\mathrm{d}c_0}{\mathrm{d}\alpha_m}\,\vline_{\,\alpha_m=\alpha_m^*}>0,
\] hence $\frac{\mathrm{d}u}{\mathrm{d}\alpha_m}\,\vline_{\,\alpha_m=\alpha_m^*}>0$.\qed
\end{proof}
Notice that $\frac{\mathrm{d}u}{\mathrm{d}\alpha_m}\,\vline_{\,\alpha_m=\alpha_m^*}>0$ implies the crossing number  at the stationary point $(x(\alpha_m^*),\,y(\alpha_m^*),\,z(\alpha_m^*))$ satisifes:
\[\gamma(x(\alpha_m^*),\,y(\alpha_m^*),\,z(\alpha_m^*),\,\alpha_m^*,\,v(\alpha_m^*))\neq
0.\] Moreover, we can check that conditions ($S1$--$S3$) for Theorem~\ref{localhopf} are satisfied. Then we have
 the following local Hopf bifurcation theorem for system~(\ref{SDDE-hes1}).
 \begin{theorem}\label{local-Hes1}
 Let  $\alpha_m^*$ be as in Lemma~\ref{lemma-4-1}. Then  system~(\ref{SDDE-hes1}) undergoes Hopf bifurcation near the stationary point $(x(\alpha_m^*),\,y(\alpha_m^*),\,z(\alpha_m^*))$ as $\alpha_m$ varies near $\alpha_m^*$.
 \end{theorem}

  \subsection{Global Hopf bifurcation}
  In this section, we develop a global Hopf bifurcation theory for  system~(\ref{SDDE-hes1}). By Lemma~\ref{sec4-eq-3} and Theorem~\ref{local-Hes1}, we know that $(x(\alpha_m^*),\,y(\alpha_m^*),\,z(\alpha_m^*))$ is the only Hopf bifurcation point and is an isolated center. To apply the global Hopf bifurcation theorem~\ref{global-new-th}, it remains to check condition (S4) and  one of the conditions among  i), ii) and iii) at Lemma~\ref{period-bound}. We first consider the boundedness of periodic solutions.
% \subsubsection{Boundedness of periodic solutions}
 \begin{theorem}\label{bounded-periodic-solutions}
 Let $(x,\,y,\,z)$ be a periodic solution of system~(\ref{SDDE-hes1}). Then   $(x,\,y,\,z)$ satisfies for every $t\in\mathbb{R}$,
 \begin{align*}
 0<x(t)\leq \frac{\alpha_m}{\mu_m},\, 0<y(t)\leq \frac{\alpha_p\alpha_m}{\mu_p\mu_m},\,0<z(t)\leq \frac{\alpha_e\alpha_p\alpha_m}{\mu_e\mu_p\mu_m}.
 \end{align*}
 \end{theorem}
 \begin{proof} Note that $h>0$ is an even integer. We have
 $\dot{x}(t)\leq -\mu_m x(t)+\alpha_m$, which by Gronwall's inequality leads to
 \begin{align}\label{bounded-1}
 x(t)\leq e^{-\mu_mt}x(0)+\frac{\alpha_m}{\mu_m}\left(1-e^{-\mu_m t}\right).
 \end{align}Since $x$ is periodic, there exists $p>0$ such that $x(t)=x(t+p)$ for every $t\in\mathbb{R}$ and for every $n\in\mathbb{N}$, we have $x(t)=x(t+np)$. Then  for every $t\in\mathbb{R}$  we have
$x(t)=x(t+np)\leq e^{-\mu_m(t+np)}x(0)+\frac{\alpha_m}{\mu_m}\left(1-e^{-\mu_m (t+np)}\right)\rightarrow  \frac{\alpha_m}{\mu_m}$ as $n\rightarrow \infty.$ Therefore, we have $x(t)\leq \frac{\alpha_m}{\mu_m}$ for every $t\in\mathbb{R}$. 

By the same token, with $x(t-\tau)\leq \frac{\alpha_m}{\mu_m}$, we obtain from the second equation of system~(\ref{SDDE-hes1}) that
$y(t)\leq \frac{\alpha_p\alpha_m}{\mu_p\mu_m}$, $t\in\mathbb{R}$, and subsequently from the third equation of system~(\ref{SDDE-hes1}) that $z(t)\leq \frac{\alpha_e\alpha_p\alpha_m}{\mu_e\mu_p\mu_m}$ for every $t\in\mathbb{R}$. 

To obtain lower bounds of $x,\,y$ and $z$, let $\bar{x}=-x$, $\bar{y}=-y$ and $\bar{z}=-z$. Then  system~(\ref{SDDE-hes1}) becomes
\begin{align}\label{SDDE-hes1-bar} 
\left\{
\begin{aligned}
\frac{\mathrm{d}\bar{x}(t)}{\mathrm{d}t} & =-\mu_m \bar{x}(t)-\frac{\alpha_m}{1+\left(\frac{\bar{z}(t-\tau)}{\bar{z}}\right)^h},\\
  \frac{\mathrm{d}\bar{y}(t)}{\mathrm{d}t} & =-\mu_p \bar{y}(t)+\alpha_p \bar{x}(t-\tau),\\
   \frac{\mathrm{d}\bar{z}(t)}{\mathrm{d}t} & =-\mu_e \bar{z}(t)+\alpha_e \bar{y}(t-\tau),\\
 \tau(t)&=c(\bar{x}(t-\tau)-\bar{x}(t)),
\end{aligned}
\right.
\end{align}
We have
 $\dot{\bar{x}}(t) < -\mu_m \bar{x}(t)$, which  leads to
 \begin{align}\label{bounded-1}
 \bar{x}(t) < e^{-\mu_mt}\bar{x}(0).
 \end{align} Note that $\bar{x}$ is also $p$-periodic. For every $t\in\mathbb{R}$  we have
\[
\bar{x}(t)=\bar{x}(t+np) < e^{-\mu_m(t+np)}\bar{x}(0)\rightarrow 0 \mbox{ as $n\rightarrow \infty.$ }
\]Therefore, we have $\bar{x}(t)\leq 0$ for every $t\in\mathbb{R}$. By the same token, with $\bar{x}(t-\tau)\leq 0$, we obtain from the second equation of system~(\ref{SDDE-hes1-bar}) that
$\bar{y}(t)\leq 0$, $t\in\mathbb{R}$, and subsequently from the third equation of system~(\ref{SDDE-hes1-bar}) that $\bar{z}(t)\leq 0$ for every $t\in\mathbb{R}$. 
Then by the definition of $(\bar{x},\,\bar{y},\,\bar{z})$, we obtain that for every $t\in\mathbb{R}$, $x(t)\geq 0,\,y(t)\geq 0,\,z(t)\geq 0.$

If there exists $t_0\in\mathbb{R}$ such that $x(t_0)=0$, then by the first equation of system~(\ref{SDDE-hes1}) we have $\dot{x}(t_0)>0$. By the continuity of $\dot{x}$, there exists $\delta>0$ such that $x$ is strictly increasing in $(t_0-\delta,\,t_0+\delta)$.
so that $x(t)<0$ for $t\in (t_0-\delta,\,t_0)$. This is a contradiction.  By the same token we have $y(t)>0$ and $z(t)>0$ for every $t\in\mathbb{R}$.\qed
 \end{proof}

 \begin{lemma}\label{sec4-lipschitz}
 Let $f_0: \mathbb{R}^3\times\mathbb{R}^3\times\mathbb{R}\rightarrow\mathbb{R}^3$ be defined by
 \[
 f_0(\theta_1,\,\theta_2)=-\left(\begin{matrix}
 \mu_m\\
 \mu_p\\
 \mu_e
 \end{matrix}\right)\cdot \theta_1+\left(\begin{matrix}
 \frac{\alpha_m}{1+\left(\frac{z_2}{\,\tilde{z}}\right)^h}\\
 \alpha_p x_2\\
 \alpha_e y_2
 \end{matrix}\right)
 \]where $\theta_1=(x_1,\,y_1,\,z_1)$ and $\theta_2=(x_2,\,y_2,\,z_2)$. Then $f_0$ is Lipschitz continuous  with a Lipschitz constant
 \[
 L_f=\max\left\{\mu_m,\,\mu_p,\,\mu_e,\,\alpha_p,\,\alpha_e,\,\frac{\alpha_mh_0}{\tilde{z}}\right\},
 \]where $h_0=\frac{h\left(1-\frac{2}{h+1}\right)^{\frac{h-1}{h}}}{\left(1+\frac{h-1}{h+1}\right)^2}.$
 \end{lemma} 
\begin{proof} We use the Mean Value theorem for integrals to obtain  a Lipschitz constant. Let $\tilde{\theta_1}=(\tilde{x}_1,\,\tilde{y}_1,\,\tilde{z}_1)$ and $\tilde{\theta_2}=(\tilde{x}_2,\,\tilde{y}_2,\,\tilde{z}_2)$. Then we have
\begin{align}\label{f-0-lip}
\left|f_0(\theta_1,\,\theta_2)-f_0(\tilde{\theta}_1,\,\tilde{\theta}_2)\right| \leq &
\max\left\{\mu_m,\,\mu_p,\,\mu_e\right\}|\theta_1-\tilde{\theta}_1|\notag\\
& +\max\left\{\alpha_p,\,\alpha_e,\,\sup_{z_2}\left|\frac{\mathrm{d}}{\mathrm{d}z_2}\frac{\alpha_m}{1+\left(\frac{z_2}{\,\tilde{z}}\right)^h}\right|\right\}|\theta_2-\tilde{\theta}_2|.
\end{align}We have
\begin{align*}
\frac{\mathrm{d}}{\mathrm{d}z_2}\frac{\alpha_m}{1+\left(\frac{z_2}{\,\tilde{z}}\right)^h}=\frac{\alpha_mh}{\tilde{z}}\frac{\left(\frac{z_2}{\,\tilde{z}}\right)^{h-1}}{\left(1+\left(\frac{z_2}{\,\tilde{z}}\right)^h\right)^2},
\end{align*}. Noticing that the map $\mathbb{R}\ni t\rightarrow \frac{t^{h-1}}{(1+t^h)^2}$ vanishes at $t=0$ and $t=\infty$ and that
\[
\frac{\mathrm{d}}{\mathrm{d}t}\frac{t^{h-1}}{(1+t^h)^2}=0,
\] if and only if $t=\pm\left(1-\frac{2}{h+1}\right)^{\frac{1}{h}}$, we obtain that
$\sup_{z_2}\left|\frac{\mathrm{d}}{\mathrm{d}z_2}\frac{\alpha_m}{1+\left(\frac{z_2}{\,\tilde{z}}\right)^h}\right|=\frac{\alpha_mh_0}{\tilde{z}}$ with \[
h_0=\frac{h\left(1-\frac{2}{h+1}\right)^{\frac{h-1}{h}}}{\left(1+\frac{h-1}{h+1}\right)^2},
\] and the supremum is achieved at $\frac{z_2}{\,\tilde{z}}=\left(1-\frac{2}{h+1}\right)^{\frac{1}{h}}$. Then by (\ref{f-0-lip})  $f_0$ is Lipschitz continuous  with a Lipschitz constant $L_f=\max\left\{\mu_m,\,\mu_p,\,\mu_e,\,\alpha_p,\,\alpha_e,\,\frac{\alpha_mh_0}{\tilde{z}}\right\}.$
\hfill\qed
\end{proof}

To apply the global Hopf bifurcation theorem,  we also use  Lemma~\ref{period-bound} to show the closure of all nontrivial periodic solutions bifurcating from the stationary point $(x(\alpha_m^*),\,y(\alpha_m^*),\,z(\alpha_m^*))$ will not include constant solution with zero period.
\begin{lemma}\label{c-1-tau} 
Let $(x,\,y,\,z)$ be a periodic solution of system~(\ref{SDDE-hes1}). If $\alpha_m<\frac{1}{c}$, then $\tau:\mathbb{R}\rightarrow\mathbb{R}$ given by $\tau(t)=c(x(t)-x(t-\tau(t)))$ exists and is continuously differentiable. 
\end{lemma}
\begin{proof} The existence and continuity of $\tau$ follows from Lemma~\ref{Lemma-2-1}.  Let $f_1: \mathbb{R}^2\rightarrow\mathbb{R}$ be defined by
\[
f_1(\tau,\,t)=\tau-c(x(t)-x(t-\tau)).
\]
Then $f_1$ is continuously differentiable with respect to $(\tau,\,t)$. Moreover, by (\ref{dot-x-1-c}) we have 
\[
\frac{\partial f_1(\tau,\,t)}{\partial \tau}=1-c\dot{x}(t-\tau).
\]
By the first equation of system~(\ref{SDDE-hes1}) and by Lemma~\ref{bounded-periodic-solutions} we have for every $t\in\mathbb{R}$, $\dot{x}(t)<\alpha_m$ and
\begin{align}\label{dot-x-1-c}
\dot{x}(t)\geq -\mu_m\frac{\alpha_m}{\mu_m}+\frac{\alpha_m}{1+\left(\frac{\alpha_e\alpha_p\alpha_m}{\mu_e\mu_p\tilde{z}}\right)^h}
>-\alpha_m.
\end{align}
Then we have $|\dot{x}|<\alpha_m$ and by (\ref{dot-x-1-c}) we have 
\[
\frac{\partial f_1(\tau,\,t)}{\partial \tau}=1-c\dot{x}(t-\tau)>0.
\]
By the Implicit Function Theorem, $\tau$ is continuously differentiable at $t\in\mathbb{R}$.\qed
\end{proof}
It follows from Lemma~\ref{c-1-tau} and ii) of Lemma~\ref{period-bound} that
if $\alpha_m<\frac{1}{c}$, then the period $\mathsf{p}$ of every nonconstant periodic solution 
satisfies $\mathsf{p}\geq\frac{4}{L_f(2+\|\dot{\tau}\|_{L^\infty})}>0$. 

Now we are in the position to state the global Hopf bifurcation theorem.

\begin{theorem} Let  $\alpha_m^*$ be as in Lemma~\ref{lemma-4-1} and $p^*=\frac{2\pi}{v^*}$ where $v^*>0$ is the imaginary part of eigenvalue of the formal linearization of  system~(\ref{SDDE-hes1}) at $\alpha_m=\alpha_m^*$. Suppose that
$\alpha_m^*<\frac{1}{c}$. There exists a connected component $\mathcal{C}$  of the closure of all nonconstant periodic solution of  system~(\ref{SDDE-hes1}) bifurcating from $(\alpha_m^*,\,p^*,\,x(\alpha_m^*),\,y(\alpha_m^*),\,z(\alpha_m^*))\in\mathbb{R}^2\times C(\mathbb{R};\mathbb{R}^3)$, which  satisfies that 
\begin{enumerate}
\item[$i)$] either the projection of $\mathcal{C}$ onto the parameter space of  the period $\mathsf{p}$ is unbounded.
\item[$i)$] or the projection of $\mathcal{C}$ onto the parameter space of $\alpha_m$ does not cross $\alpha=0$ but is not contained in any compact subset of the interval $(0,\,\frac{1}{c})$;

\end{enumerate}
\end{theorem}
\begin{proof}We first show that if $\alpha_m=0$, system~(\ref{SDDE-hes1}) has no nonconstant periodic solutions. Otherwise, let $(x,\,y,\,z)$ be a  nonconstant periodic solution with $\alpha_m=0$. Then from  system~(\ref{SDDE-hes1})  $\dot{x}=-\mu_mx$ implies that $x=0$ and subsequently $y=z=0$. This is a contradiction.
\end{proof}

In the following we consider $\alpha_m$ in $(0,\,\frac{1}{c})$ and introduce the following change of variables:
\begin{align}
\alpha_m=q(\alpha)=\frac{2}{c\pi}\left(\arctan \alpha-\frac{\pi}{2}\right)+\frac{1}{c},
\end{align}
where $q$ is an increasing function of $\alpha$ with $\lim_{\alpha\rightarrow-\infty}q(\alpha)=0$
 and $\lim_{\alpha\rightarrow+\infty}q(\alpha)=\frac{1}{c}$. Then system~(\ref{SDDE-hes1}) is rewritten as
 \begin{align}\label{SDDE-hes2} 
\left\{
\begin{aligned}
\frac{\mathrm{d}x(t)}{\mathrm{d}t} & =-\mu_m x(t)+\frac{q(\alpha)}{1+\left(\frac{z(t-\tau)}{\tilde{z}}\right)^h},\\
  \frac{\mathrm{d}y(t)}{\mathrm{d}t} & =-\mu_p y(t)+\alpha_p x(t-\tau),\\
   \frac{\mathrm{d}z(t)}{\mathrm{d}t} & =-\mu_e z(t)+\alpha_e y(t-\tau),\\
 \tau(t)&=c(x(t)-x(t-\tau)),
\end{aligned}
\right.
\end{align} with $\alpha\in\mathbb{R}$ and $\alpha^*=q^{-1}(\alpha_m^*)$ the critical value of $\alpha$ for a unique Hopf bifurcation point. By Theorem~\ref{local-Hes1}
There exists a connected component $\mathcal{C}_0$  of the closure of all nonconstant periodic solution of  system~(\ref{SDDE-hes2}) bifurcating from the stationary point $(\alpha^*,\,p^*,\,x(\alpha^*),\,y(\alpha^*),\,z(\alpha^*))\in\mathbb{R}^2\times C(\mathbb{R};\mathbb{R}^3)$.

By Lemma~\ref{sec4-lipschitz}, condition (S4) is satisfied by system~(\ref{SDDE-hes2}). By  Lemma~\ref{c-1-tau}, the function $\tau$ defined by $\tau(t)=c(x(t)-x(t-\tau(t)))$ for  a  nonconstant periodic solution $(x,\,y,\,z)$ of system~(\ref{SDDE-hes2}) is continuously differentiable. Hence by Lemma~\ref{period-bound}, the period $\mathsf{p}$ of every nonconstant periodic solution $(x,\,y,\,z)$ of system~(\ref{SDDE-hes2}) is positive. Notice that $(\alpha^*,\,p^*,\,x(\alpha^*),\,y(\alpha^*),\,z(\alpha^*))$ is the only bifurcation point of system~(\ref{SDDE-hes2}), by Theorem~\ref{global-new-th}, the connected component $\mathcal{C}_0$ is unbounded in $\mathbb{R}^2\times C(\mathbb{R};\mathbb{R}^3)$.

Notice that by Theorem~\ref{bounded-periodic-solutions}, the projection of $\mathcal{C}_0$ onto the space of $(x,\,y,\,z)\in C(\mathbb{R};\mathbb{R}^3)$ is bounded. The unboundedness of $\mathcal{C}_0$ is either because of the unbounded projection onto the parameter space of  the period $\mathsf{p}$, or the projection of $\mathcal{C}$ onto the parameter space of $\alpha$.

Notice that $q$ induce a homeomorphism $(q,\,id): \mathbb{R}^2\times C(\mathbb{R};\mathbb{R}^3)\rightarrow \mathbb{R}^2\times C(\mathbb{R};\mathbb{R}^3)$ defined by
\[
(q,\,id)(\alpha,\,h)=(q(\alpha),\,h).
\]The image $\mathcal{C}=(q,\,id)(\mathcal{C}_0)$ of $\mathcal{C}_0$ under  $(q,\,id)$
is a connected component  of the closure of all nonconstant periodic solution of  system~(\ref{SDDE-hes1}) bifurcating from the bifurcation point \[(\alpha_m^*,\,p^*,\,x(\alpha_m^*),\,y(\alpha_m^*),\,z(\alpha_m^*))\in\mathbb{R}^2\times C(\mathbb{R};\mathbb{R}^3),\] which  satisfies that either the projection of $\mathcal{C}$ onto the parameter space of  the period $\mathsf{p}$ is unbounded, or the projection of $\mathcal{C}$ onto the parameter space of $\alpha_m$ does not cross the hyperplane $\alpha_m=0$ but is not contained in any compact subset of the interval $(0,\,\frac{1}{c})$. \qed

\section{Concluding remarks}\label{Conclude}
Motivated by the extended Goodwin's model with a state-dependent delay governed by an algebraic equation, we developed a global Hopf bifurcation theory for differential-algebraic equations with state-dependent delay, using the $S^1$-equivariant degree. This is based on the framework described in \cite{MR2644135} where the technique of formal linearization is employed to obtain auxiliary linear systems at the stationary states which indicate local and global Hopf bifurcation using a homotopy argument. We remark that the local and global Hopf bifurcation theory we developed for system~(\ref{SDDE-general}) is also applicable for systems with the delay  given by $\tau(t)=h(x_t)$ where $h$  is a function of $x_t(s)=x(t+s), -r\leq s\leq 0,\,r>0$, provided that $\tau$ is continuous.

The local and global Hopf bifurcation theories are applied to the extended Goodwin's model which describes intracellular processes in the genetic regulatory dynamics. We obtained two alternatives for the connected component $\mathcal{C}$ of periodic solutions in the Fuller space $\mathbb{R}^2\times C(\mathbb{R};\mathbb{R}^3)$. Namely, the projection of $\mathcal{C}$ onto   the parameter space of  the period $\mathsf{p}$ is unbounded, or the projection  onto the parameter space of $\alpha_m$   is not contained in any compact subset of the interval $(0,\,\frac{1}{c})$. We remark that in the previous case, there exists a sequence of periodic solutions with periods going to $\infty$. From (\ref{abs-global-hopf}), system~(\ref{SDDE-hes1}) can be represented as
\begin{align*}
\frac{2\pi}{\mathsf{p}}\frac{\mathrm{d}x}{\mathrm{d}t}=\mathscr{N}_0(x,\,\alpha_m,\,
{\mathsf{p}}),\,\mathsf{p}>0,
\end{align*}
where $x$ is normalized to be $2\pi$-periodic. Notice from the definition of $\mathscr{N}_0$ at (\ref{N-0-definition}) that $\mathsf{p}$ appears only in the time domain of $\mathscr{N}_0$.   Note also that the periodic solutions are uniformly bounded with $\alpha_m\in (0,\,\frac{1}{c})$. Then with $\mathsf{p}\rightarrow\infty$, this alternative implies the possibility that the system has a sequence of nonconstant periodic solutions with the limiting profile satisfying the algebraic equation $\mathscr{N}_0(x,\,\alpha_m,\,\mathsf{p})=0.$ See (\cite{mallet1992boundary},\cite{mallet1996boundary},\cite{mallet2003boundary}) for a discussion of limiting profiles for differential equations with state-dependent delays.

If the projection of $\mathcal{C}$ onto the parameter space of  the period $\mathsf{p}$ is bounded, we have the latter alternative that the projection of $\mathcal{C}$ onto the parameter space of  the period $\alpha_m$ is not  contained in any compact subset of the interval $(0,\,\frac{1}{c})$. Since $\mathcal{C}$ will not cross the hyperplane $\alpha_m=0$, and will not blow up at $\alpha_m=\frac{1}{c}$ with the boundedness of the solutions and periods, $\mathcal{C}$  must cross the hyperplane  $\alpha_m=\frac{1}{c}$ leaving the solutions at $\alpha_m\geq \frac{1}{c}$ out of the scope of the discussion.

We also remark that the state-dependent delay in system~(\ref{SDDE-general}) may be negative or positive and is not {\it a priori} advanced or retarded type delay differential equations. It remains open to investigate this type of systems  in  general settings for a qualitative theory including existence and uniqueness of solutions. For systems with mixed type constant delays, see, among many others, \cite{Rustichini, mallet1999mixed}.
 \bibliographystyle{siam}

\begin{thebibliography}{10}

\bibitem{HU-JDE-1}
{\sc Balanov, Z., Hu, Q. and Krawcewicz, W.},
\newblock Global Hopf bifurcation of differential equations with threshold-type state-dependent delay, 
\newblock{\em J. Differential Equations}, 
257 (2014),  2622--2670.  


\bibitem{cooke1996problem}
{\sc Cooke, K., and Huang, W.},
\newblock On the problem of linearization for state-dependent delay
  differential equations.
\newblock {\em Proceedings of the American Mathematical Society 124}, 5 (1996),
  1417--1426.

%
%\bibitem{Driver-1}
%{\sc Driver, R.\,D.,}
%\newblock A neutral system with state-dependent delay \newblock {\em Journal of Differential Equations}, 54(1984), 73 -- 86.

\bibitem{Goodwin}
{\sc Goodwin, B.,} 
\newblock  Oscillatory behavior in enzymatic control processes, 
\newblock {\em Advances in Enzyme Regulation}, 3 (1965),   425--438.

\bibitem{HKWW}
{\sc Hartung, F., Krisztin, T., Walther, H.-O., and Wu, J.},
\newblock Chapter 5: {F}unctional {D}ifferential {E}quations with
  {S}tate-{D}ependent {D}elays: {T}heory and {A}pplications.
\newblock In {\em Handbook of Differential Equations: Ordinary Differential
  Equations}, P.~D. {A. Ca{\~N}ada} and A.~Fonda, Eds., vol.~3. North-Holland,
  2006, pp.~435--545.

\bibitem{Hu-diffusion}
{\sc Hu, Q.,}
\newblock A model of regulatory dynamics with threshold type state-dependent delay,
\newblock submitted, 2017.

%
%\bibitem{Rolling-Hu}
%{\sc Hu, Q.,} 
%\newblock A model of cold metal rolling processes with state-dependent delay, 
%\newblock{\em SIAM J. Appl. Math.}, 76 (2016):1076-1100.
%
%\bibitem{MR3022268}
%{\sc Hu, Q., Krawcewicz, W., and Turi, J.},
%\newblock Global stability lobes of turning processes with state-dependent
%  delay.
%\newblock {\em SIAM J. Appl. Math. 72}, 5 (2012), 1383--1405.
%
%\bibitem{MR2888330}
%{\sc Hu, Q., Krawcewicz, W., and Turi, J.},
%\newblock Stabilization in a state-dependent model of turning processes.
%\newblock {\em SIAM J. Appl. Math. 72}, 1 (2012), 1--24.
%
%\bibitem{MR2665435}
%{\sc Hu, Q., and Wu, J.},
%\newblock Global continua of rapidly oscillating periodic solutions of
%  state-dependent delay differential equations.
%\newblock {\em J. Dynam. Differential Equations 22}, 2 (2010), 253--284.

\bibitem{MR2644135}
{\sc Hu, Q., and Wu, J.},
\newblock Global {H}opf bifurcation for differential equations with
  state-dependent delay.
\newblock {\em J. Differential Equations 248}, 12 (2010), 2801--2840.
%
%\bibitem{MR3023381}
%{\sc Hu, Q., Wu, J., and Zou, X.},
%\newblock Estimates of periods and global continua of periodic solutions for
%  state-dependent delay equations.
%\newblock {\em SIAM J. Math. Anal. 44}, 4 (2012), 2401--2427.

%\bibitem{MR2993842}
%{\sc Hu, Q., and Zhao, X.-Q.},
%\newblock Global dynamics of a state-dependent delay model with unimodal
%  feedback.
%\newblock {\em J. Math. Anal. Appl. 399}, 1 (2013), 133--146.


  
\bibitem{Turi-1}
{\sc Insperger, T., St{\'e}p{\'a}n, G., and Turi, J.},
\newblock State-dependent delay in regenerative turning processes.
\newblock {\em Nonlinear Dyn. 47\/} (2007), 275--283.
  
 
 
\bibitem{kw}\label{kw} 
{\sc Krawcewicz, W. and Wu, J.},
\newblock Theory of degrees with applications to bifurcations and differential equations,
\newblock {\em Canadian Mathematical Society Series of Monographs and Advanced Texts},
John Wiley \& Sons, New York, 1997.


 \bibitem{MY}\label{MY}
 Mallet-Paret, J. and  Yorke, A. J., Snakes: oriented families of periodic orbits, their sources, sinks, and continuation,
 J. Differential Equations \textbf{43} (1982),  419--450. 
 
 \bibitem{mallet1992boundary}
{\sc Mallet-Paret, J., and Nussbaum, R.}
\newblock Boundary layer phenomena for differential-delay equations with
  state-dependent time lags, {I}.
\newblock {\em Archive for rational mechanics and analysis 120}, 2 (1992),
  99--146.

\bibitem{mallet1996boundary}
{\sc Mallet-Paret, J., and Nussbaum, R.}
\newblock Boundary layer phenomena for differential-delay equations with state
  dependent time lags: {II}.
\newblock {\em Journal fur die Reine und Angewandte Mathematik\/} (1996),
  129--197.

\bibitem{mallet2003boundary}
{\sc Mallet-Paret, J., and Nussbaum, R.}
\newblock Boundary layer phenomena for differential-delay equations with
  state-dependent time lags: {III}.
\newblock {\em Journal of Differential Equations 189}, 2 (2003), 640--692.

\bibitem{mallet1999mixed}
{\sc Mallet-Paret, J.}
\newblock The Fredholm alternative for functional differential equations of mixed type.
\newblock{\em J. Dynamics and Differential Equations}, 11 (1999), 1--47.

\bibitem{Rustichini} {\sc Rustichini, A.} 
 \newblock Hopf bifurcation for functional differential equations of mixed type. 
 \newblock{J. Dynamics and Differential Equations},
1 (1989), 145--177.
 
 
\bibitem{Smith}
{\sc Smith, H.,}
\newblock Structured population models, threshold-type delay and
  functional-differential equations.
\newblock {\em Delay and differential equations, AMS, IA, 1991\/} (1992), 52 --
  64.
 
  
  \bibitem{Walther-echo}
{\sc Walther, H.,}
\newblock  Stable periodic motion of a system using echo for position control,
 \newblock{\em  Journal of Dynamics and Differential Equations}
15 (2003),   no. 1, 143--223

 
\bibitem{Walther-semiflow}
{\sc Walther, H.,}
\newblock Smoothness properties of semiflows for differential equations with state-dependent delays, 
\newblock {\em Journal of Mathematical Sciences} 
   124 (2004), 5193--5207.
 

\bibitem{Smith19931}
{\sc Smith, H.~L.,}
\newblock Reduction of structured population models to threshold-type delay
  equations and functional differential equations: a case study.
\newblock {\em Mathematical Biosciences 113}, 1 (1993), 1 -- 23.

\bibitem{Vidossich}
{\sc Vidossich, V.},
\newblock On the structure of periodic solutions of differential equations, 
\newblock{\em J. Differential Equations} 21 (1976), 263 -- 278.


\bibitem{Wu-1} \label{Wu-1} {\sc Wu, J.}, Global continua of periodic solutions to some
differential equations of neutral type, T\^{o}hoku Math J.,
\textbf{45} (1993), 67-88.

\end{thebibliography}

\end{document}